\documentclass[a4,12pt]{amsart}
\usepackage{amsmath,amssymb,amsthm} 
\usepackage[all]{xy} 
\usepackage{enumerate} 
\usepackage{indentfirst} 
\usepackage{soul}
\numberwithin{equation}{section} 
\usepackage[pagebackref=true]{hyperref} 
\renewcommand*{\backref}[1]{} 
\renewcommand*{\backrefalt}[4]{
  \ifcase #1 
    (not cited)
  \or
    (cited on p.~#2)
  \else
    (cited on pp.~#2)
  \fi}
\def\Hom{\operatorname{Hom}} 
\def\End{\operatorname{End}} 
\def\soc{\operatorname{soc}} 
\def\top{\operatorname{top}} 
\def\Im{\operatorname{Im}} 
\def\GL{\operatorname{GL}} 
\def\Stab{\operatorname{Stab}} 
\newcommand{\Mat}{\operatorname{Mat}}
\def\add{\operatorname{\mathsf{add}}} 
\def\mod{\operatorname{\mathsf{mod}}} 
\def\Gr{\mathrm{Gr}} 
\def\Aut{\mathrm{Aut}} 
\newtheorem{theorem}{Theorem}[section] 
\newtheorem{theorem*}{Theorem} 
\newtheorem{corollary}[theorem]{Corollary} 
\newtheorem{corollary*}[theorem*]{Corollary} 
\newtheorem{lemma}[theorem]{Lemma} 
\newtheorem{proposition}[theorem]{Proposition} 
\theoremstyle{definition} 
\newtheorem{definition}[theorem]{Definition} 
\newtheorem{remark}[theorem]{Remark} 
\newtheorem*{question*}{Question} 
\newtheorem{conjecture}[theorem]{Conjecture} 
\newtheorem*{conjecture*}{Conjecture} 
\newtheorem{example}[theorem]{Example} 
\newtheorem*{notation*}{Notation} 
\newtheorem*{claim*}{Claim} 

\begin{document}
\setlength{\baselineskip}{17pt}

\title{On bricks of one-point extension algebras}
\begin{abstract}
Let $\Lambda=B[E]$ be the one-point extension algebra of $B$ by the extension module $E$. Using the standard triple description of $\Lambda$-modules, we characterize mixed bricks through an injectivity condition and a scalar-stabilizer condition. Under the assumption that $B$ is $E$-visible finite, we interpret mixed bricks as orbits in Grassmannians and derive a necessary and sufficient criterion for the brick-finiteness of $\Lambda$.
\end{abstract}

\author{Qi Wang}
\email{wang2025@dlut.edu.cn (cc:infinite-wang@outlook.com)}
\address{School of Mathematical Sciences, Dalian University of Technology, Dalian City, Liaoning Province 116024, China}

\thanks{2020 {\em Mathematics Subject Classification.}
Primary 16G20; Secondary 16G60, 14M15}

\keywords{One-point extensions, brick-finiteness, Grassmannians, brick--Brauer--Thrall conjecture}

\date{July 31, 2026.}
\maketitle

\section{Introduction}
Bricks, also known as Schur representations, have long occupied a fundamental position in the representation theory of finite-dimensional algebras. They may be regarded as relative analogues of simple modules: bricks occur as simple objects in wide subcategories, and filtrations by suitably ordered bricks provide a way of decomposing modules that is adapted to Hom-orthogonality and torsion-theoretic structures; see \cite{Asai-semibricks, BTZ-brick, Ringel-bricks, Ringel-brick-chain-filtrations}. Their role became particularly prominent with the development of $\tau$-tilting theory \cite{Adachi-Iyama-Reiten-2014}. A fundamental result in \cite{Demonet-Iyama-Jasso-2019} characterizes $\tau$-tilting finite algebras in terms of the finiteness of bricks and connects this condition with the structure of torsion classes in the module category. Hence, brick-finiteness is not merely a restricted form of representation-finiteness, but a natural finiteness property lying at the intersection of $\tau$-tilting theory, torsion theory, silting theory, etc. These developments have also led to brick analogues of the classical Brauer--Thrall problems; see \cite{Mousavand-minimal-tau-infinite, Mousavand-Paquette-survey, Schroll-Treffinger-first-BT, Schroll-Treffinger-Valdivieso-band}. Although brick-finiteness has been determined for various important classes of algebras, there is currently no general and computationally effective criterion for deciding whether an arbitrary finite-dimensional algebra is brick-finite.

In this paper, we study brick-finiteness for one-point extension algebras. Throughout, $\Bbbk$ denotes an algebraically closed field, and $\mod A$ denotes the category of finite-dimensional right $A$-modules for a finite-dimensional $\Bbbk$-algebra $A$. Let $B$ be a finite-dimensional $\Bbbk$-algebra and let $E={}_{\Bbbk}E_B$ be a finite-dimensional $\Bbbk$-$B$-bimodule. We denote by $\Lambda=B[E]$ the one-point extension of $B$ by $E$, and we write $1_{\Lambda}=e_B+e_0$, where
\[
B[E]:=
\begin{pmatrix}
B&0\\
E&\Bbbk
\end{pmatrix},
\quad
e_B:=
\begin{pmatrix}
1_B&0\\
0&0
\end{pmatrix},
\quad
e_0:=
\begin{pmatrix}
0&0\\
0&1
\end{pmatrix}.
\]
One-point extensions provide one of the most basic procedures for constructing new algebras from known ones: at the level of ordinary quivers, they amount to adjoining a new source together with arrows determined by the top of $E$. They have been extensively used in the study of representation type, Auslander--Reiten components, tilted and quasitilted algebras, and related enlargement procedures; see, for instance, \cite{de-la-Pena-one-point,Happel-Slungard-one-point,SS-Vol-3}. From a computational point of view, it is therefore natural to ask whether the brick-finiteness of $\Lambda$ can be detected from the representation theory of $B$ and from concrete linear-algebraic data attached to $E$. The purpose of this paper is to develop such a framework.

Every right $\Lambda$-module can be described by a triple $M=(N,W,\eta)$, where $N\in\mod B$, $W$ is a finite-dimensional $\Bbbk$-vector space, and $\eta\colon W\longrightarrow\Hom_B(E,N)$ is a $\Bbbk$-linear map. If $W=0$, then $(N,0,0)$ is a brick in $\mod \Lambda$ if and only if $N$ is a brick in $\mod B$. Besides, the only brick with $N=0$ is the simple module $(0,\Bbbk,0)$ supported at the extension vertex. Thus, the new bricks arising from the one-point extension are the \emph{mixed bricks}, namely, the bricks $(N,W,\eta)$ with $N\neq0$ and $W\neq0$. Our first result gives a concrete criterion for a mixed module to be a brick.

\begin{theorem*}[Proposition~\ref{prop::mixed-brick-criterion}]
Let $M=(N,W,\eta)\in\mod \Lambda$ with $N\neq0$ and $W\neq0$. Then, $M$ is a mixed brick if and only if $\eta$ is injective and
\[
\Stab_{\End_B(N)}(\Im\eta)
:=
\left\{
f\in\End_B(N)
\ \middle|\
f_*(\Im\eta)\subseteq\Im\eta
\right\}
=
\Bbbk\operatorname{id}_{N},
\]
where $f_*\colon\Hom_B(E,N)\longrightarrow\Hom_B(E,N)$ is the map induced by post-composition.
\end{theorem*}

To make this criterion computationally effective, we call an indecomposable $B$-module $X$ \emph{$E$-visible} if $\Hom_B(E,X)\neq0$, and we say that $B$ is \emph{$E$-visible finite}\footnote{Several constructions developed in this paper can also be formulated when $B$ is not $E$-visible finite. However, in order to keep the notation uniform, we work throughout under the assumption that $B$ is $E$-visible finite.} if there are only finitely many isomorphism classes of $E$-visible indecomposable $B$-modules. Let $X_1,X_2,\ldots,X_t$ be a complete list of representatives of the isomorphism classes of $E$-visible indecomposable $B$-modules. It follows from Corollary~\ref{cor::mixed-brick-summand} that the $B$-module part $N$ of every mixed brick $M=(N,W,\eta)\in\mod \Lambda$ belongs to $\add(X_1\oplus X_2\oplus\cdots\oplus X_t)$. Hence, by the Krull--Schmidt theorem, 
\[
N\cong N_{\mathbf{m}} := \bigoplus_{i=1}^tX_i^{m_i}
\]
for some nonzero multiplicity vector $\mathbf{m}=(m_1,m_2,\ldots,m_t)\in\mathbb{Z}_{\geq 0}^t$. We define 
\[
h(\mathbf{m}) :=\dim_{\Bbbk}\Hom_B(E,N_{\mathbf{m}})\quad \text{and} \quad r(\mathbf{m}) := \dim_{\Bbbk}\End_B(N_{\mathbf{m}}).
\]

For a nonzero vector $\mathbf{m}\in\mathbb{Z}_{\geq 0}^t$ and an integer $1\leq\ell\leq h(\mathbf{m})$, consider the Grassmannian $\Gr_\ell\bigl(\Hom_B(E,N_{\mathbf{m}})\bigr)$. The algebraic group $\Aut_B(N_{\mathbf{m}})$ acts on this Grassmannian by post-composition. After identifying $W$ with $\Im\eta$, mixed bricks are controlled by pairs $(N,L)$ with $0\neq L\subseteq\Hom_B(E,N)$ such that the only endomorphisms of $N$ preserving $L$ are scalar multiples of the identity. We then define
\[
\mathcal{L}(\mathbf{m},\ell)
:=
\left\{
L\in\Gr_\ell\bigl(\Hom_B(E,N_{\mathbf{m}})\bigr)
\ \middle|\
\Stab_{\End_B(N_{\mathbf{m}})}(L)
=
\Bbbk\operatorname{id}_{N_{\mathbf{m}}}
\right\}.
\]
Our second result gives the geometric interpretation of mixed bricks.

\begin{theorem*}[Theorem~\ref{thm:dim-conditions}, Lemma~\ref{lem::Lml-open} and Proposition~\ref{prop::mixed-bricks-orbits}]
Suppose that $B$ is $E$-visible finite, $0\neq\mathbf{m}\in\mathbb{Z}_{\geq 0}^t$ and $1\leq\ell\leq h(\mathbf{m})$. Then, the following statements hold.
\begin{enumerate}

\item If $\mathcal{L}(\mathbf{m},\ell)\neq\varnothing$, then $r(\mathbf{m})-1
\leq
\ell\bigl(h(\mathbf{m})-\ell\bigr)$. 

\item The subset $\mathcal{L}(\mathbf{m},\ell) \subseteq \Gr_\ell\bigl(\Hom_B(E,N_{\mathbf{m}})\bigr)$ is Zariski-open.

\item The isomorphism classes of mixed bricks $M=(N,W,\eta)$ satisfying $N\cong N_{\mathbf{m}}$ and $\dim_{\Bbbk}W=\ell$ are in bijection with the $\Aut_B(N_{\mathbf{m}})$-orbits in $\mathcal{L}(\mathbf{m},\ell)$.
\end{enumerate}
\end{theorem*}

We call a pair $(\mathbf{m},\ell)$ \emph{admissible} if $r(\mathbf{m})-1 \leq \ell\bigl(h(\mathbf{m})-\ell\bigr)$. An admissible pair is called \emph{critical} if the equality holds and \emph{strict} otherwise. We say that $(\mathbf{m},\ell)$ is \emph{realized} if $\mathcal{L}(\mathbf{m},\ell)\neq\varnothing$. For every $L\in\mathcal{L}(\mathbf{m},\ell)$, the left-hand side of the inequality is the dimension of the orbit $\Aut_B(N_{\mathbf{m}})\cdot L$, whereas the right-hand side is the dimension of the Grassmannian. The critical and strict admissible pairs exhibit sharply different geometric
behavior.

\begin{theorem*}[Theorem~\ref{thm::critical-pair-one-orbit}, Propositions~\ref{prop:strict-infinite-bricks}
and~\ref{prop:strict-second-bbt}]
Suppose that $B$ is $E$-visible finite, and let $(\mathbf{m},\ell)$ be an admissible pair.
\begin{enumerate}
\item If $(\mathbf{m},\ell)$ is critical, then $\mathcal{L}(\mathbf{m},\ell)$ is either empty or a single $\Aut_B(N_{\mathbf{m}})$-orbit. Consequently, a critical admissible pair is realized by at most one isomorphism class of mixed bricks in $\mod \Lambda$.

\item If $(\mathbf{m},\ell)$ is strict and realized, then $\mathcal{L}(\mathbf{m},\ell)$ contains infinitely many $\Aut_B(N_{\mathbf{m}})$-orbits. Consequently, $\Lambda$ admits infinitely many pairwise nonisomorphic mixed bricks realizing $(\mathbf{m},\ell)$, all having the same dimension vector $\bigl(\underline{\dim}_B N_{\mathbf{m}},\ell\bigr)$. 
\end{enumerate}
\end{theorem*}

We obtain a necessary and sufficient condition for the finiteness of mixed bricks as below.

\begin{theorem*}[Theorem~\ref{thm:mixed-finite}]
Suppose that $B$ is $E$-visible finite. Then, the following statements are equivalent.
\begin{enumerate}
\item There are only finitely many isomorphism classes of mixed bricks in $\mod \Lambda$.

\item There are only finitely many realized admissible pairs, and all are critical.
\end{enumerate}
If these equivalent conditions hold, then $\Lambda$ is brick-finite if and only if $B$ is brick-finite.
\end{theorem*}

Now the problem of determining brick-finiteness for a one-point extension separates naturally into two parts: a numerical problem governed by the functions $h(\mathbf{m})$ and $r(\mathbf{m})$, and a geometric realization problem governed by the Zariski-open subsets $\mathcal{L}(\mathbf{m},\ell)$ of the corresponding Grassmannians. In particular, every realized strict admissible pair produces infinitely many pairwise nonisomorphic mixed bricks with one fixed dimension vector, as predicted by the second brick--Brauer--Thrall conjecture.

Geometric ideas of this kind have also appeared in several earlier approaches to the representation theory of finite-dimensional algebras. Framed representation moduli have been realized as Grassmannians of submodules \cite{FedotovFramedModuli}, while Grassmannian parameter spaces have been used more broadly in the classification of modules \cite{HuisgenZimmermannGrassmannians}. Particularly relevant is the study of moduli spaces of stable representations of one-point extensions \cite{DonmezReinekeOnePointModuli}, among other related works.

The paper is organized as follows. In Section~\ref{sec::preliminaries}, we recall the required facts on Zariski topology, projective algebraic varieties, Grassmannians, and one-point extension algebras. In Section~\ref{Section-3}, we introduce $E$-visible modules, admissible pairs, and the numerical invariants $h(\mathbf{m})$ and $r(\mathbf{m})$. We prove the dimension inequality for mixed bricks, study the geometry of the loci $\mathcal{L}(\mathbf{m},\ell)$, establish the critical--strict dichotomy, and characterize the finiteness of admissible pairs in terms of the positivity of a quadratic form. In Section~\ref{Section-4}, we apply these results to the brick-finiteness of one-point extensions. We give examples showing that strict admissible pairs need not be realized, derive structural restrictions on the extension module $E$, and discuss the connection with the second brick--Brauer--Thrall conjecture. In Section~\ref{Section-5}, we develop four applications of the general framework. We treat one-point extensions by projective modules and by simple modules, investigate the behavior of semibricks and epibricks, and conclude with an almost gentle example illustrating how our method can determine brick-finiteness beyond the scope of the available criteria for gentle and special biserial algebras.

\section{Preliminaries}\label{sec::preliminaries}
We work with a finite-dimensional basic $\Bbbk$-algebra $A$ over an algebraically closed field $\Bbbk$, and all modules are right modules unless otherwise stated. Let $\{e_i\mid 1\leq i\leq n\}$ be a complete set of primitive orthogonal idempotents of $A$, and set $P_i:=e_iA$. Thus, $P_i$ is the indecomposable projective right $A$-module associated with $e_i$. With this convention, we have a natural identification $\Hom_A(P_i,P_j)\cong e_jAe_i$. Consequently, an irreducible homomorphism $P_i\to P_j$ corresponds to an arrow $j\to i$ in the Gabriel quiver of $A$. More generally, homomorphisms in $\Hom_A(P_i,P_j)$ are represented by linear combinations of paths from $j$ to $i$, modulo the defining relations of $A$. If $f\colon P_i\to P_j$ and $g\colon P_j\to P_k$ correspond to paths from $j$ to $i$ and from $k$ to $j$, respectively, then $g\circ f\colon P_i\to P_k$ corresponds to the concatenated path from $k$ to $i$. On the other hand, when we view a right $A$-module $M$ as a representation of the quiver, an arrow $\alpha\colon i\to j$ acts in the usual direction by right multiplication $Me_i\to Me_j$. 

\subsection{Zariski topology}\label{subsec::Zariski-topology}
We refer to \cite[Chapter I]{Hartshorne-book-1977} for the standard background on affine algebraic sets and Zariski topology. A \emph{topology} on a set $\mathcal{X}$ may be defined by specifying a collection $\mathcal{C}$ of subsets of $\mathcal{X}$ satisfying the following conditions: 
\begin{enumerate}
\item $\varnothing$ and $\mathcal{X}$ belong to $\mathcal{C}$; 

\item if $\{\mathcal{X}_i\}_{i\in I}$ is any family of subsets in $\mathcal{C}$, then $\cap_{i\in I}\mathcal{X}_i\in\mathcal{C}$; 

\item if $\mathcal{X}_1,\mathcal{X}_2,\ldots,\mathcal{X}_s\in\mathcal{C}$, then $\mathcal{X}_1\cup \mathcal{X}_2\cup\cdots\cup \mathcal{X}_s\in\mathcal{C}$. 
\end{enumerate}
We call the pair $(\mathcal{X},\mathcal{C})$ a topological space. If the collection $\mathcal{C}$ is clear from the context, we simply refer to $\mathcal{X}$ as a topological space. Each subset of $\mathcal{X}$ belonging to $\mathcal{C}$ is called a \emph{closed set}, while a subset $\mathcal{U}\subseteq \mathcal{X}$ is called an \emph{open set} if its complement $\mathcal{X}\setminus \mathcal{U}$ is closed. A subset $\mathcal{Y}\subseteq \mathcal{X}$ is called \emph{locally closed} if
there exist an open subset $\mathcal{U}\subseteq \mathcal{X}$ and a closed subset $\mathcal{V}\subseteq \mathcal{X}$ such that $\mathcal{Y}=\mathcal{U}\cap \mathcal{V}$. One may equivalently define a topology by specifying its open subsets. In that case, the open subsets must contain $\varnothing$ and $\mathcal{X}$, must be closed under arbitrary unions, and must be closed under finite intersections. These two definitions are equivalent, because closed subsets are exactly complements of open subsets. 

A typical example of a topological space is $\mathbb{R}$ with its usual topology. In this case, a subset is closed if its complement is a union of open intervals. For instance, $(0,1)$ is open, while $[0,1]$ is closed since $\mathbb{R}\setminus[0,1]=(-\infty,0)\cup(1,\infty)$. The subset $(0,1]$ is neither open nor closed, while the subsets $\varnothing$ and $\mathbb{R}$ are both open and closed. 

A subset $\mathcal{U}$ of a topological space $\mathcal{X}$ is called \emph{dense} if every nonempty open subset of $\mathcal{X}$ has a nonempty intersection with $\mathcal{U}$. For example, $\mathbb{Q}$ is dense in $\mathbb{R}$ with respect to the usual topology. 
More generally, $\mathbb{Q}^m$ is dense in $\mathbb{R}^m$ for any positive integer $m$.

A topological space $\mathcal{X}$ is called \emph{compact} if every open cover of $\mathcal{X}$ has a finite subcover. In the usual topology on $\mathbb{R}^m$, the Heine--Borel Theorem (see \cite[Theorem 2.41]{Rudin-book-1976}) says that a subset of $\mathbb{R}^m$ is compact if and only if it is closed and bounded. For instance, a closed interval $[a,b]\subseteq\mathbb{R}$ is compact, while $(a,b)$ and $\mathbb{R}$ are not compact. We will use the Extreme Value Theorem: a continuous real-valued function on a compact set reaches its maximum and minimum, see \cite[Theorem~4.16]{Rudin-book-1976}. 

A topological space $\mathcal{X}$ is said to be \emph{irreducible} if $\mathcal{X}$ is nonempty and cannot be written as the union of two proper closed subsets. Equivalently, $\mathcal{X}$ is irreducible if any two nonempty open subsets of $\mathcal{X}$ have a nonempty intersection. It is clear that a nonempty open subset of an irreducible topological space is again irreducible. 

If $\mathcal{X}$ is a topological space, we define the \emph{dimension} of $\mathcal{X}$ to be the supremum of all integers $r\geq0$ for which there exists a strictly increasing chain of irreducible closed subsets $\mathcal{X}_0\subsetneq \mathcal{X}_1\subsetneq\cdots\subsetneq \mathcal{X}_r\subseteq \mathcal{X}$.

\begin{proposition}[{\cite{Hartshorne-book-1977}}]\label{prop::dimension-open-cover}
Let $\mathcal{X}$ be a Noetherian topological space, and suppose that $\mathcal{X}$ is covered by a family of open subsets $\{\mathcal{U}_i\}_{i\in I}$. Then, $\dim \mathcal{X}=\sup_{i\in I}\dim \mathcal{U}_i$, where each $\mathcal{U}_i$ is endowed with the induced topology from $\mathcal{X}$.
\end{proposition}

We now recall the Zariski topology on affine space. Let 
\[
\mathbb{A}^n:=\Bbbk^n=\{(a_1,a_2,\ldots,a_n)\mid a_i\in\Bbbk\}
\]
be the affine $n$-space over $\Bbbk$. Given an ideal $\mathfrak{I}\subseteq \Bbbk[x_1,x_2,\ldots,x_n]$, let $f_1,f_2,\ldots,f_t$ be generators of $\mathfrak{I}$ (since $\Bbbk[x_1,x_2,\ldots,x_n]$ is a Noetherian ring, the ideal $\mathfrak{I}$ has a finite set of generators). We define
\[
\mathcal{Z}(\mathfrak{I}) := \left\{
(a_1,a_2,\ldots,a_n)\in\mathbb{A}^n
\ \middle|\
f_i(a_1,a_2,\ldots,a_n)=0,\ 1\leq i\leq t
\right\}.
\]
Equivalently,
\[
\mathcal{Z}(\mathfrak{I}) = \left\{
\mathbf{a}\in\mathbb{A}^n
\ \middle|\
f(\mathbf{a})=0 \text{ for all } f\in \mathfrak{I}
\right\}.
\]
Thus, $\mathcal{Z}(\mathfrak{I})$ is the set of common zeros of all polynomials in $\mathfrak{I}$. This set is called an \emph{affine algebraic subset} of $\mathbb{A}^n$. The \emph{Zariski topology} on $\mathbb{A}^n$ is defined as the topology whose closed sets are the affine algebraic subsets $\mathcal{Z}(\mathfrak{I})$, where $\mathfrak{I}$ runs through the ideals of $\Bbbk[x_1,x_2,\ldots,x_n]$. We give some examples. 

\begin{example}
(1) In $\mathbb{A}^1=\Bbbk$, we have $\mathcal{Z}(x)=\{0\}$. More generally, if $f(x)\in\Bbbk[x]$ is a nonzero polynomial, then $\mathcal{Z}(f(x))$ is the finite set of roots of $f(x)$. Hence, in the Zariski topology on $\mathbb{A}^1$, the proper closed subsets are finite sets. It follows that a nonempty open subset of $\mathbb{A}^1$ is obtained by removing finitely many points. 

(2) In $\mathbb{A}^2$, the closed subset $\mathcal{Z}(y) = \{(a,b)\in\mathbb{A}^2\mid b=0\}$ is the $x$-axis. Similarly, $\mathcal{Z}(x) = \{(a,b)\in\mathbb{A}^2\mid a=0\}$ is the $y$-axis. The closed subset $\mathcal{Z}(xy)$ is the union of the two coordinate axes, namely $\mathcal{Z}(xy)=\mathcal{Z}(x)\cup \mathcal{Z}(y)$. Hence, $\mathcal{Z}(xy)$ is reducible. 

(3) The subset $\mathcal{Z}(y-x^2)\subseteq \mathbb{A}^2$ is the affine parabola $\{(a,b)\in\mathbb{A}^2\mid b=a^2\}$. It is closed in the Zariski topology because it is defined by one polynomial equation.
\end{example}

Let $\mathcal{Z}$ be an affine algebraic subset of $\mathbb{A}^n$. We always endow $\mathcal{Z}$ with the Zariski topology induced from $\mathbb{A}^n$. It follows that a subset $\mathcal{U}\subseteq \mathcal{Z}$ is open in $\mathcal{Z}$ if and only if $\mathcal{U}= \mathcal{Z}\cap \mathcal{U}'$ for some Zariski-open subset $\mathcal{U}'\subseteq\mathbb{A}^n$. According to \cite[Chapter I, Proposition 1.5]{Hartshorne-book-1977}, $\mathcal{Z}$ is a finite union of irreducible closed subsets, and the maximal ones with respect to inclusion are called \emph{irreducible components} of $\mathcal{Z}$. An irreducible closed subset of $\mathbb{A}^n$, equipped with the induced Zariski topology, is called an \emph{affine algebraic variety}. 

Let $\mathcal{Z}$ be an affine algebraic subset with irreducible components $\mathcal{Z}_1,\mathcal{Z}_2,\ldots, \mathcal{Z}_s$. We have $\dim  \mathcal{Z}=\max\{\dim  \mathcal{Z}_i\mid 1\leq i\leq s\}$. In particular, $\dim\mathbb{A}^n=n$. 

If $f\in\Bbbk[x_1,x_2,\ldots,x_n]$, then the complement of $\mathcal{Z}(f)$ is denoted by
\[
\mathcal{D}(f):=\mathbb{A}^n\setminus \mathcal{Z}(f)
=
\{\mathbf{a}\in\mathbb{A}^n\mid f(\mathbf{a})\neq0\}.
\]
The subsets $\mathcal{D}(f)$ are called \emph{principal open subsets}. They are basic examples of Zariski-open subsets. We shall frequently use the following elementary principle: a condition given by the non-vanishing of a polynomial defines a Zariski-open subset. 

For later use, we record a standard example involving matrix ranks. 
\begin{example}\label{ex::open-set}
Let $\Mat_{m\times n}(\Bbbk)\cong \mathbb{A}^{mn}$ be the affine space of all $m\times n$ matrices over $\Bbbk$. For an integer $r\geq0$, the subset $\{M\in\Mat_{m\times n}(\Bbbk)\mid \operatorname{rank}M\leq r\}$ is Zariski-closed. Indeed, if $r<\min\{m,n\}$, the condition $\operatorname{rank}M\leq r$ is equivalent to saying that all $(r+1)\times(r+1)$ minors of $M$ are zero, and these minors are polynomial functions in the entries of $M$. If $r\geq\min\{m,n\}$, the subset is the whole space, and hence is also Zariski-closed. On the other hand, the subset $\{M\in\Mat_{m\times n}(\Bbbk)\mid \operatorname{rank}M\geq r\}$ is Zariski-open. Indeed, for $1\leq r\leq\min\{m,n\}$, this condition is equivalent to saying that at least one $r\times r$ minor of $M$ is nonzero. If $r=0$, the subset is the whole space, while if $r>\min\{m,n\}$, the subset is empty; in both cases, it is Zariski-open. 

We will also use the following equivalent form. Let $\mathcal{Z}$ be an affine algebraic subset, which we regard as a parameter space. For each point $\mathbf{a}\in  \mathcal{Z}$, suppose that we are given a linear map $\Phi_{\mathbf{a}}\colon \Bbbk^n\to\Bbbk^m$. Suppose that, after choosing bases, the entries of the matrix of $\Phi_{\mathbf{a}}$ are polynomial functions on $\mathcal{Z}$. Then, for every integer $d$, the subset
\[
 \mathcal{Z}_{\leq d} := \left\{
\mathbf{a}\in  \mathcal{Z}
\ \middle|\
\dim_{\Bbbk}\ker\Phi_{\mathbf{a}}\leq d
\right\}
\]
is Zariski-open in $\mathcal{Z}$, due to $\dim_{\Bbbk}\ker\Phi_{\mathbf{a}} = n-\operatorname{rank}\Phi_{\mathbf{a}}$. 
\end{example}

\subsection{Projective algebraic variety}
We briefly recall the projective setting. The \emph{projective $n$-space} over $\Bbbk$ is defined by
\[
\mathbb{P}^n:=(\mathbb{A}^{n+1}\setminus\{0\})/\Bbbk^{\times},
\]
where $\Bbbk^{\times}$ acts on $\mathbb{A}^{n+1}\setminus\{0\}$ by scalar multiplication. A point of $\mathbb{P}^n$ is written as $[a_0:a_1:\cdots:a_n]$ with $(a_0,a_1,\ldots,a_n)\neq0$, and $[a_0:a_1:\cdots:a_n] = [\lambda a_0:\lambda a_1:\cdots:\lambda a_n]$ 
for any $\lambda\in\Bbbk^{\times}$. A polynomial $f\in \Bbbk[x_0,x_1,\ldots,x_n]$ is called \emph{homogeneous} if all monomials appearing in $f$ have the same total degree. If $f$ is homogeneous, then the condition $f(a_0,a_1,\ldots,a_n)=0$ is independent of the chosen representative of the projective point $[a_0:a_1:\cdots:a_n]$.

Given a homogeneous ideal $\mathfrak{J}\subseteq \Bbbk[x_0,x_1,\ldots,x_n]$, we define
\[
\mathcal{Z}_{\mathbb{P}}(\mathfrak{J})
:=
\left\{
[a_0:a_1:\cdots:a_n]\in\mathbb{P}^n
\ \middle|\
f(a_0,a_1,\ldots,a_n)=0 \text{ for all homogeneous } f\in \mathfrak{J}
\right\},
\]
which is called a \emph{projective algebraic subset} of $\mathbb{P}^n$. The Zariski topology on $\mathbb{P}^n$ is defined by declaring $\mathcal{Z}_{\mathbb{P}}(\mathfrak{J})$ to be the closed subsets. An irreducible projective algebraic subset, endowed with the induced Zariski topology, is called a \emph{projective algebraic variety}.

If $f\in\Bbbk[x_0,x_1,\ldots,x_n]$ is homogeneous, then $\mathbb{P}^n\setminus \mathcal{Z}_{\mathbb{P}}(f)$ is Zariski-open in $\mathbb{P}^n$. It follows that a condition given by the non-vanishing of a homogeneous polynomial defines a Zariski-open subset in projective space.

We recall that $\mathbb{P}^n$ is covered by standard open subsets. For $0\leq i\leq n$, set
\[
\mathcal{U}_i:=\{[a_0:a_1:\cdots:a_n]\in\mathbb{P}^n\mid a_i\neq0\}.
\]
Then, $\mathcal{U}_i$ is Zariski-open in $\mathbb{P}^n$, and $\mathcal{U}_i\cong\mathbb{A}^n$. For example,
\[
\mathcal{U}_0\longrightarrow\mathbb{A}^n,\qquad
[a_0:a_1:\cdots:a_n]\longmapsto
\left(\frac{a_1}{a_0},\ldots,\frac{a_n}{a_0}\right)
\]
is an isomorphism, with inverse $\mathbb{A}^n\rightarrow \mathcal{U}_0$ given by $(a_1',\ldots,a_n')\mapsto [1:a_1':\cdots:a_n']$. Hence, $\mathbb{P}^n=\mathcal{U}_0\cup \mathcal{U}_1\cup\cdots\cup \mathcal{U}_n$. It follows from Proposition \ref{prop::dimension-open-cover} that $\dim \mathbb{P}^n=n$.

We end this subsection by recalling the following standard fact; for the affine case, see \cite[Chapter~I, Proposition~1.10]{Hartshorne-book-1977}. 

\begin{proposition}\label{prop:open-subset-dim}
Let $\mathcal{Z}$ be an irreducible affine or projective algebraic variety. Then, $\dim \mathcal{Z}=\dim \mathcal{U}$ for every nonempty open subset $\mathcal{U}$ of $\mathcal{Z}$.
\end{proposition}

\subsection{Grassmannian}
We recall some basic definitions and properties of Grassmannians from \cite[Chapter 5]{LB-Grassmannian-book}. Let $V$ be an $n$-dimensional $\Bbbk$-vector space with $n\geq 1$. For an integer $0\leq \ell\leq n$, the \emph{Grassmannian}, or Grassmann variety, is defined by
\[
\Gr_\ell(V) := \bigl\{ L\subseteq V \;\big|\; \dim_{\Bbbk} L = \ell \bigr\},
\]
namely, $\Gr_\ell(V)$ is the set of all $\ell$-dimensional linear subspaces of $V$.

\begin{example}
Let us recall some basic examples of Grassmannians. 
\begin{enumerate}
\item If $\ell=0$, then $\Gr_0(V)=\{0\}$ is a single point. 

\item If $\ell=1$, then $\Gr_1(V)=\mathbb{P}(V)$ is the projective space associated with $V$. In particular, $\Gr_1(\Bbbk^2)=\mathbb{P}^1$ is a projective line.

\item For any $\ell$, there is a canonical isomorphism $\Gr_\ell(V)\cong\Gr_{n-\ell}(V^*)$ given by sending a subspace $L\subseteq V$ to its annihilator $L^\perp = \{ f\in V^* \mid f(L)=0 \}$. 
\end{enumerate}
\end{example}

The Grassmannian $\Gr_\ell(V)$ has the structure of a projective algebraic variety via the \emph{Pl\"ucker embedding}
\[
\Gr_\ell(V) \hookrightarrow \mathbb{P}\Bigl(\bigwedge^{\ell} V\Bigr),\quad
L \longmapsto [v_1\wedge\cdots\wedge v_{\ell}],
\]
where $\{v_1,\dots,v_{\ell}\}$ is a basis of $L$. This map is well-defined, because replacing the basis of $L$ changes $v_1\wedge\cdots\wedge v_{\ell}$ only by a nonzero scalar. Its image is a closed subvariety of $\mathbb{P}(\wedge^{\ell}V)$ defined by the quadratic Pl\"ucker relations; see \cite[Theorem 5.2.3]{LB-Grassmannian-book}. We always endow $\Gr_\ell(V)$ with the Zariski topology induced by this embedding. The general linear group $\GL(V)$ acts transitively on $\Gr_\ell(V)$, since any two $\ell$-dimensional subspaces can be mapped to each other by an invertible linear transformation. Fix a subspace $L_0\in\Gr_\ell(V)$ and let
\[
\Stab_{\GL(V)}(L_0) := \{ g\in\GL(V) \mid gL_0 = L_0 \}
\]
be its stabilizer. Then, $\Stab_{\GL(V)}(L_0)$ is a parabolic subgroup of $\GL(V)$. The map
\[
\GL(V)/\Stab_{\GL(V)}(L_0) \longrightarrow \Gr_\ell(V),\qquad g\Stab_{\GL(V)}(L_0) \longmapsto gL_0
\]
is an isomorphism of algebraic varieties. It follows that $\Gr_\ell(V)$ is irreducible and homogeneous. Hence, $\Gr_\ell(V)$ is a projective algebraic variety. 

We recall the well-known dimension formula for Grassmannians; see, for instance, \cite[Corollary 5.3.8]{LB-Grassmannian-book}. For the sake of completeness, we include a proof. 

\begin{proposition}\label{prop::grassmannian-dimension}
Suppose that $n\geq 1$ and $0\leq \ell\leq n$. Then, $\dim \Gr_\ell(V)=\ell(n-\ell)$.
\end{proposition}
\begin{proof}
If $\ell=0$ or $\ell=n$, then $\Gr_\ell(V)$ consists of a single point. In both cases, $\ell(n-\ell)=0$, so the formula holds. We treat the case $0<\ell<n$ in the following. Fix a decomposition $V=V'\oplus V''$ with $\dim_{\Bbbk}V' =\ell$ and $\dim_{\Bbbk} V''=n-\ell$. Consider the subset
\[
\mathcal{L} =\left\{
L\in \Gr_\ell(V)
\ \middle|\
\pi_{V'}|_L\colon L\to V'\text{ is an isomorphism}
\right\},
\]
where $\pi_{V'}\colon V\to V'$ is the projection onto $V'$. After choosing bases of $V'$ and $V''$, a subspace $L$ may be represented by an $n\times \ell$ matrix. The map $\pi_{V'}|_L$ is an isomorphism if and only if the $\ell\times\ell$ minor given by the $V'$-coordinate rows is nonzero. Hence, $\mathcal{L}$ is Zariski-open in $\Gr_\ell(V)$. 

We now identify $\mathcal{L}$ with an affine space. Let $L\in \mathcal{L}$. Since $\pi_{V'}|_L\colon L\to V'$ is an isomorphism, for each $u\in V'$ there exists a unique vector $x_u\in L$ such that $\pi_{V'}(x_u)=u$. Write $x_u=u+a(u)$ with $a(u)\in V''$. Then, the assignment $u\mapsto a(u)$ gives a $\Bbbk$-linear map $a\colon V' \to V''$. Moreover, $L=\{u+a(u)\mid u\in V'\}$. Conversely, every linear map $a\colon V' \to V''$ defines a subspace
\[
\Gamma(a)=\{u+a(u)\mid u\in V'\}\subseteq V.
\]
If $u+a(u)=0$, then $u=0$ by the direct sum decomposition $V=V'\oplus V''$. Hence, $\dim_{\Bbbk} \Gamma(a)=\dim_{\Bbbk} V'=\ell$. The restriction $\pi_{V'}|_{\Gamma(a)}\colon \Gamma(a)\to V'$ is an isomorphism, and therefore $\Gamma(a)\in \mathcal{L}$. Thus, we obtain a bijection $\Hom_{\Bbbk}(V',V'') \rightarrow \mathcal{L}$ given by $a\mapsto \Gamma(a)$. In coordinates, this map and its inverse are given by regular maps. Therefore, $\mathcal{L}\cong \Hom_{\Bbbk}(V',V'')$ as algebraic varieties. Since $\Hom_{\Bbbk}(V',V'')\cong \mathbb{A}^{\ell(n-\ell)}$, we have $\dim\mathcal{L}=\ell(n-\ell)$. 

It remains to explain why these open subsets determine the dimension of the whole Grassmannian. Choose a basis of $V$. For each $\ell$-element subset $I\subseteq \{1,\ldots,n\}$, let $\mathcal{L}_I\subseteq \Gr_\ell(V)$ be the set of $\ell$-dimensional subspaces whose Pl\"ucker coordinate corresponding to $I$ is nonzero. Equivalently, $\mathcal{L}_I$ is the set of subspaces whose projection onto the coordinate subspace spanned by the basis vectors indexed by $I$ is an isomorphism. Each $\mathcal{L}_I$ is a Zariski-open subset of $\Gr_\ell(V)$, and $\mathcal{L}_I\cong \mathbb{A}^{\ell(n-\ell)}$. We have
\[
\Gr_\ell(V)=\bigcup_{|I|=\ell}\mathcal{L}_I.
\]
Using Proposition \ref{prop::dimension-open-cover}, the dimension of $\Gr_\ell(V)$ is $\ell(n-\ell)$. 
\end{proof}

\subsection{One-point extension}
We recall the basic construction of one-point extension algebras; for more details, see \cite[Chapter~IV]{SS-Vol-3}. Let $\Lambda=B[E]$. Every right $\Lambda$-module can be described as a triple $(N,W,\eta)$, where $N\in\mod B$, $W$ is a finite-dimensional $\Bbbk$-vector space, and $\eta\colon W\to \Hom_B(E,N)$ is a $\Bbbk$-linear map. 
\begin{itemize}
\item[($\Rightarrow$)]
Let $M$ be a right $\Lambda$-module, and set $N:=Me_B$ and $W:=Me_0$. Then $N$ is naturally a right $B$-module, while $W$ is naturally a finite-dimensional $\Bbbk$-vector space. For $w\in W$ and $x\in E$, we rewrite $x$ as an element in $e_0\Lambda e_B$ and define
\[
wx := w
\begin{pmatrix}
0&0\\
x&0
\end{pmatrix}
\in (Me_0)(e_0\Lambda e_B)\subseteq Me_B =N.
\] 
This gives a $\Bbbk$-bilinear map $W\times E\to N$, $(w,x)\mapsto wx$. Equivalently, we obtain a right $B$-module homomorphism $\theta\colon W\otimes_{\Bbbk}E\to N$, given by $\theta(w\otimes x)=wx$. 

\item[($\Leftarrow$)] Let $N\in\mod B$, $W$ be a finite-dimensional $\Bbbk$-vector space, and $\theta\colon W\otimes_{\Bbbk}E\to N$ be a right $B$-module homomorphism. Then, $N\oplus W$ becomes a right $\Lambda$-module by the action
\[
(n,w)
\begin{pmatrix}
b&0\\
x&\lambda
\end{pmatrix}
=
\bigl(nb+\theta(w\otimes x),\,w\lambda\bigr),
\]
for all $n\in N$, $w\in W$, $b\in B$, $x\in E$, and $\lambda\in\Bbbk$. Here, the $\Bbbk$-linearity and right $B$-linearity of $\theta$ ensure that this action is associative.
\end{itemize}
By the tensor--Hom adjunction, we have a natural isomorphism
\[
\Hom_B(W\otimes_{\Bbbk}E,N)
\cong
\Hom_{\Bbbk}\bigl(W,\Hom_B(E,N)\bigr).
\] 
Hence, the map $\theta$ above is equivalently given by a $\Bbbk$-linear map $\eta\colon W\to \Hom_B(E,N)$. Explicitly, the map $\eta(w)\colon E\to N$ is given by $\eta(w)(x)=\theta(w\otimes x)$ for any $w\in W$, $x\in E$. Therefore, right $\Lambda$-modules are equivalently described by triples $(N,W,\eta)$ as above.

We then recall the following lemma from \cite[Chapter~IV, Lemma~1.8]{SS-Vol-3}. 
\begin{lemma}\label{lem::hom-non-zero}
Let $M=(N,W,\eta)$ be an indecomposable $\Lambda$-module with $\eta\neq 0$. Then, $\Hom_B(E,X)\neq 0$ for every nonzero indecomposable direct summand $X$ of $N$ in $\mod B$. 
\end{lemma}

A right $\Lambda$-module $M=(N,W,\eta)$ falls into one of the following two cases. If $\eta=0$, then $M\simeq (N,0,0)\oplus(0,W,0)$. In particular, if $M$ is indecomposable, then either $M\simeq (N,0,0)$ with $N$ indecomposable in $\mod B$, or $M\simeq (0,\Bbbk,0)$. If $\eta\neq0$, then necessarily $N\neq0$ and $W\neq0$. Moreover, if the indecomposable $B$-modules satisfying $\Hom_B(E,X)\neq0$ are precisely $X_1,\ldots,X_t$, and if $M$ is indecomposable, then Lemma~\ref{lem::hom-non-zero} implies that $N\in\add(X_1\oplus\cdots\oplus X_t)$. We call a module $M=(N,W,\eta)$ \emph{mixed} if $N\neq0$ and $W\neq0$. Thus, for an indecomposable module, being mixed is equivalent to $\eta\neq0$. In particular, a \emph{mixed brick} is a brick $M=(N,W,\eta)$ with $N\neq0$ and $W\neq0$, equivalently, a brick with $\eta\neq0$.

This motivates us to introduce the following definition.
\begin{definition}
Let $E\in\mod B$. We say that an indecomposable $B$-module $X$ is \emph{$E$-visible} if $\Hom_B(E,X)\neq0$, and that $B$ is \emph{$E$-visible finite} if there are only finitely many isomorphism classes of $E$-visible indecomposable $B$-modules. Otherwise, $B$ is said to be \emph{$E$-visible infinite}.
\end{definition}

It follows that $B$ is representation-finite if and only if $B$ is $S_i$-visible finite for all $1\leq i\leq m$, where $S_1,\ldots,S_m$ form a complete list of simple $B$-modules. In general, for a fixed choice of $E$, representation-finite always implies $E$-visible finite, but the converse does not hold. For example, let $B=\Bbbk(0\longrightarrow 1 \rightrightarrows 2)$ and let $E$ be the simple $B$-module corresponding to the vertex $0$. Then, the only indecomposable $B$-module visible from $E$ is $E$ itself, but $B$ is representation-infinite. The relation between $E$-visible finite and brick-finite is weaker: the above example also shows that $E$-visible finite does not imply brick-finite; every brick-finite but representation-infinite algebra gives, for a suitable simple module $E$, an example showing that brick-finite does not imply $E$-visible finite.

It is worth mentioning that $E$-visible finiteness of $B$, together with brick-finiteness of $B$, does not imply brick-finiteness of $\Lambda=B[E]$. For instance, if $B=\Bbbk$ and $E=\Bbbk^2$, then $B$ is $E$-visible finite and brick-finite, but $\Lambda$ is the Kronecker algebra, hence brick-infinite. 

Let $M=(N,W,\eta)$ and $M'=(N',W',\eta')$ be two right $\Lambda$-modules. A morphism $M\to M'$ is given by a pair $(f,g)$, where $f\colon N\to N'$ is a right $B$-module homomorphism and $g\colon W\to W'$ is a $\Bbbk$-linear map, satisfying $f_*\circ\eta=\eta'\circ g$. Here, $f_*\colon \Hom_B(E,N)\to\Hom_B(E,N')$ is the $\Bbbk$-linear map induced by post-composition, namely $f_*(\varphi)=f\circ\varphi$, for $\varphi\in \Hom_B(E,N)$. 

Suppose $h:=\dim_{\Bbbk} \Hom_B(E,N)$. Since invertibility is determined by the non-vanishing of the determinant, the automorphism group $\Aut_B(N)$ is a Zariski-open subset of $\End_B(N)$, and then $\Aut_B(N)$ is an algebraic group. We note that $\Aut_B(N)$ acts on $\Hom_B(E,N)$ by post-composition, and hence acts on the Grassmannian
\[
\Gr_{\ell}(\Hom_B(E,N)) =\bigl\{
L\subseteq \Hom_B(E,N) \mid \dim_{\Bbbk}L=\ell\bigr\},
\quad 1\leq \ell\leq h,
\]
by $f\cdot L:=f_*(L)$. We show how this action controls isomorphism classes of $\Lambda$-modules with a fixed $B$-module part.

\begin{proposition}\label{prop::iso-orbit}
Let $M=(N,W,\eta)$ and $M'=(N,W',\eta')$ be two $\Lambda$-modules such that $\eta$, $\eta'$ are injective and $\dim_{\Bbbk}W=\dim_{\Bbbk}W'=\ell$. Then, $M$ and $M'$ are isomorphic as $\Lambda$-modules if and only if $\Im\eta$ and $\Im\eta'$ lie in the same $\Aut_B(N)$-orbit in $\Gr_{\ell}(\Hom_B(E,N))$.
\end{proposition}
\begin{proof}
Suppose that $(f,g)\colon M\to M'$ is an isomorphism. Then, $f\in\Aut_B(N)$ and $g\colon W\to W'$ is a $\Bbbk$-linear isomorphism.
The compatibility condition $f_*\circ\eta=\eta'\circ g$ implies $f_*(\Im\eta)=\Im\eta'$. Thus, $\Im\eta$ and $\Im\eta'$ lie in the same $\Aut_B(N)$-orbit.

Conversely, suppose that $f\in\Aut_B(N)$ satisfies $f_*(\Im\eta)=\Im\eta'$. Since $\eta$ and $\eta'$ are injective, the map
\[
g:=(\eta')^{-1}\circ f_*|_{\Im\eta}\circ\eta
\colon W\to W'
\]
is a well-defined $\Bbbk$-linear isomorphism. Moreover, it satisfies $f_*\circ\eta=\eta'\circ g$. Hence $(f,g)$ is an isomorphism $M\to M'$.
\end{proof}

\section{Mixed bricks}\label{Section-3}
We start with the following easy observation.
\begin{lemma}\label{lem:eta-injective}
If $M=(N,W,\eta)$ is a mixed brick in $\mod \Lambda$, then the map $\eta\colon W\to \Hom_B(E,N)$ is injective. 
\end{lemma}
\begin{proof}
Suppose, for a contradiction, that $\ker\eta\neq 0$. Choose a nonzero element $w_0\in\ker\eta$ and a nonzero $\Bbbk$-linear functional $\lambda\colon W\to\Bbbk$. We define a $\Bbbk$-linear map $g\colon W\to W$ by $g(w)=\lambda(w)w_0$ for any $w\in W$. Since $\lambda\neq 0$ and $w_0\neq 0$, the map $g$ is nonzero. Moreover, $\Im g\subseteq \Bbbk w_0\subseteq\ker\eta$, and hence $\eta\circ g=0$.
 
Let $f=0\colon N\to N$ be the zero $B$-module homomorphism. Then, $f_*\circ\eta=0=\eta\circ g$, namely, $(f,g)$ defines an endomorphism of $M$. This endomorphism is nonzero because $g\neq 0$. On the other hand, it cannot be an automorphism, since its $N$-component is the zero map and $N\neq 0$. This contradicts the assumption that $M$ is a brick.
\end{proof}

We obtain a useful restriction on the $B$-module part of a mixed brick. 
\begin{corollary}\label{cor::mixed-brick-summand}
Let $M=(N,W,\eta)\in\mod \Lambda$ be a mixed brick. Then, every nonzero indecomposable direct summand $X$ of $N$ satisfies $\Hom_B(E,X)\neq 0$. 
\end{corollary}
\begin{proof}
Since $M$ is a mixed brick, it is indecomposable. Using Lemma \ref{lem:eta-injective}, $\eta$ is injective and hence nonzero. Then, the statement follows immediately from Lemma \ref{lem::hom-non-zero}.
\end{proof}

Using Lemma~\ref{lem:eta-injective}, we obtain the following criterion. 
\begin{proposition}\label{prop::mixed-brick-criterion}
Let $M=(N,W,\eta)\in\mod \Lambda$ with $N\neq 0$ and $W\neq 0$. Then, $M$ is a mixed brick if and only if $\eta$ is injective and
\begin{equation}\label{equ::mixed-brick-criterion}
\Stab_{\End_B(N)}(\Im\eta):=\left\{
f\in\End_B(N)\ \middle|\ f_*(\Im\eta)\subseteq \Im\eta
\right\}
=
\Bbbk\operatorname{id}_{N} .
\end{equation}
\end{proposition}
\begin{proof}
Suppose that $\eta$ is injective. Via the isomorphism $\eta\colon W\xrightarrow{\sim}\Im\eta\subseteq \Hom_B(E,N)$, the module $M$ is isomorphic to  $(N,\Im\eta)$ with the natural inclusion $\Im\eta\hookrightarrow \Hom_B(E,N)$. Equivalently, $(N,\Im\eta)$ has the following right $\Lambda$-module structure:
\[
(n,\varphi)
\begin{pmatrix}
b&0\\
x&\lambda
\end{pmatrix}
=
\bigl(nb+\varphi(x),\,\varphi \lambda\bigr),
\qquad
n\in N,\ \varphi\in \Im\eta,\ b\in B,\ x\in E,\ \lambda\in\Bbbk .
\]
Let $(f,g)$ be an endomorphism of $(N,\Im\eta)$, where $f\colon N\to N$ and $g\colon \Im\eta\to \Im\eta$ are $\Bbbk$-linear maps. By the commutativity with the $\Lambda$-action, we obtain $f(nb)=f(n)b$, and hence $f\in\End_B(N)$. Moreover, we obtain $f(\varphi(x))=g(\varphi)(x)$, equivalently, $f\circ \varphi=g(\varphi)$. It follows that every endomorphism of $M$ is uniquely determined by an element $f\in\End_B(N)$ satisfying $f_*(\Im\eta)\subseteq\Im\eta$, where $f_*$ is defined on $\Hom_B(E,N)$ and $g=f_*|_{\Im\eta}$. Conversely, every $f\in\End_B(N)$ satisfying $f_*(\Im\eta)\subseteq\Im\eta$ determines an endomorphism of $(N,\Im\eta)$ by taking $g=f_*|_{\Im\eta}$. Hence, there is an algebra isomorphism
\[
\End_{\Lambda}(M)\cong \End_{\Lambda}(N,\Im\eta)
\cong
\left\{
f\in\End_B(N)\ \middle|\ f_*(\Im\eta)\subseteq \Im\eta
\right\}.
\] 

If $M$ is a mixed brick, the map $\eta$ is injective by Lemma~\ref{lem:eta-injective}. Using the algebra isomorphism above, the condition that $M$ is a brick is therefore equivalent to \eqref{equ::mixed-brick-criterion}. Conversely, if $\eta$ is injective and \eqref{equ::mixed-brick-criterion} holds, then the algebra isomorphism above gives $\End_{\Lambda}(M)=\Bbbk\operatorname{id}_M$, and hence $M$ is a mixed brick. This proves the assertion.
\end{proof}

Up to replacing $W$ by $\Im\eta$, mixed bricks in $\mod \Lambda$ are controlled by pairs $(N,L)$ with $N\in\mod B$ and $0\neq L\subseteq\Hom_B(E,N)$ such that $\Stab_{\End_B(N)}(L) = \Bbbk\operatorname{id}_{N}$. Moreover, by Proposition~\ref{prop::iso-orbit}, we have $(N,L) \cong (N',L')$ if and only if there exists $f\colon N \xrightarrow{\sim} N'$ satisfying $f_*(L) = L'$.

\subsection{Necessary conditions for mixed bricks}
Recall that $\Lambda=B[E]$. In this subsection, we derive some necessary conditions for mixed bricks in $\mod \Lambda$.

\begin{lemma}\label{lem:zero-action}
Suppose that $\Hom_B(E,N)\neq0$. If there exists a nonzero endomorphism $f\in\End_B(N)$ which acts as zero on $\Hom_B(E,N)$ by post-composition, then no mixed brick in $\mod \Lambda$ has $B$-module part isomorphic to $N$.
\end{lemma}
\begin{proof}
Suppose, for a contradiction, that $M=(N,W,\eta)$ is a mixed brick. By Proposition~\ref{prop::mixed-brick-criterion}, the map $\eta$ is injective and $\Stab_{\End_B(N)}(\Im\eta)=\Bbbk\operatorname{id}_{N}$. Since $f$ acts as zero on $\Hom_B(E,N)$, we have $f_*(\Im\eta)=0\subseteq \Im\eta$. 
Hence, $f\in\Stab_{\End_B(N)}(\Im\eta)$. 
Moreover, $f$ cannot be a nonzero scalar multiple of $\operatorname{id}_{N}$, because $\Hom_B(E,N)\neq0$ and every nonzero scalar acts nontrivially on $\Hom_B(E,N)$. 
Thus, $f$ is a nonzero non-scalar element of $\Stab_{\End_B(N)}(\Im\eta)$, a contradiction. 
\end{proof}

It is then reasonable to introduce the following definition.
\begin{definition}\label{def::E-faithful-module}
Let $N\in \mod B$. We say that $N$ is \emph{$E$-faithful} if the natural representation
\[
\begin{aligned}
\rho_N\colon \End_B(N)&\longrightarrow \End_{\Bbbk}\Hom_B(E,N),\\
f&\longmapsto f_*,
\qquad
f_*(\varphi):=f\circ\varphi,\; \forall \; \varphi\in\Hom_B(E,N)
\end{aligned}
\]
is injective.
\end{definition}

\begin{proposition}
Let $M=(N,W,\eta)$ be a mixed brick in $\mod \Lambda$. Then, $N$ is $E$-faithful. 
\end{proposition}
\begin{proof}
By Lemma~\ref{lem:eta-injective}, the map $\eta$ is injective, and then $0\neq \Im\eta\subseteq \Hom_B(E,N)$. By Proposition~\ref{prop::mixed-brick-criterion}, we have $\Stab_{\End_B(N)}(\Im\eta)=\Bbbk\operatorname{id}_{N}$. Let $f\in\ker\rho_N$. Then, $f_*$ acts as zero on $\Hom_B(E,N)$, and hence $f_*(\Im\eta)=0\subseteq \Im\eta$. It follows that $f\in\Stab_{\End_B(N)}(\Im\eta)$, i.e., $f=\lambda\operatorname{id}_{N}$ for some $\lambda\in\Bbbk$. Since $f_*$ acts as zero on the nonzero vector space $\Hom_B(E,N)$, we must have $\lambda=0$. Hence, $\ker\rho_N=0$, and $N$ is $E$-faithful.
\end{proof}

In the following, we assume that $B$ is $E$-visible finite and derive some necessary dimension conditions for mixed bricks in $\mod \Lambda$. Let $X_1,X_2,\ldots,X_t$ be a complete list of indecomposable $E$-visible $B$-modules. If no such indecomposable modules exist, then Corollary~\ref{cor::mixed-brick-summand} implies that there are no mixed bricks in $\mod \Lambda$; hence, in the following, we assume that $t\geq1$. By the Krull--Schmidt theorem and Corollary~\ref{cor::mixed-brick-summand}, if $(N,W,\eta)\in \mod \Lambda$ is a mixed brick, then
\[
N\cong N_{\mathbf{m}}:= \bigoplus_{i=1}^t X_i^{m_i}
\]
for some nonzero multiplicity vector $\mathbf{m}:=(m_1,\ldots,m_t)\in\mathbb{Z}_{\geq 0}^t$. 

We set $d_i:=\dim_{\Bbbk}\Hom_B(E,X_i)$ for each $1\leq i\leq t$. Since $\Hom_B(E,N_{\mathbf{m}}) \cong \bigoplus_{i=1}^t \Hom_B(E,X_i)^{m_i}$, we define
\[
h(\mathbf{m})
:=
\dim_{\Bbbk}\Hom_B(E,N_{\mathbf{m}})
=
\sum_{i=1}^t m_i d_i .
\] 

We define the Hom matrix $\mathsf{C}=(c_{ij})_{1\leq i,j\leq t}$ by $c_{ij}:=\dim_{\Bbbk}\Hom_B(X_i,X_j)$. As a $\Bbbk$-vector space, we have 
\[
\End_B(N_{\mathbf{m}})
\cong
\bigoplus_{1\leq i,j\leq t}
\Hom_B(X_i^{m_i},X_j^{m_j})\cong
\bigoplus_{1\leq i,j\leq t}
\Hom_B(X_i,X_j)^{m_im_j},  
\]
and then set
\[
r(\mathbf{m}) := \dim_{\Bbbk}\End_B(N_{\mathbf{m}})
= \sum_{1\leq i,j\leq t} m_i m_j c_{ij}
= \mathbf{m}^{\mathsf{T}}\mathsf{C}\mathbf{m}.
\]

\begin{theorem}\label{thm:dim-conditions}
Let $M=(N,W,\eta)$ be a mixed brick in $\mod \Lambda$ with $N\cong N_{\mathbf{m}}$ and 
$\ell=\dim_{\Bbbk}W$. Then, 
\[
r(\mathbf{m})-1 \leq \ell\bigl(h(\mathbf{m})-\ell\bigr).
\]
In particular,
\[
r(\mathbf{m})-1
\leq
\left\lfloor \frac{h(\mathbf{m})^2}{4}\right\rfloor.
\]
\end{theorem}
\begin{proof}
By Lemma~\ref{lem:eta-injective}, $\eta$ is injective.
Hence, $0\neq \Im\eta\in \Gr_{\ell}(\Hom_B(E,N))$. Recall from Proposition~\ref{prop::mixed-brick-criterion} that $\Stab_{\End_B(N)}(\Im\eta)=\Bbbk\operatorname{id}_{N}$. Now consider the stabilizer of $\Im\eta$ inside the automorphism group:
\[
\Stab_{\Aut_B(N)}(\Im\eta):=
\left\{
f\in\Aut_B(N) \ \middle|\ f_*(\Im\eta)=\Im\eta
\right\}.
\]
Since $f\in\Aut_B(N)$ is invertible, $f_*(\Im\eta)=\Im\eta$ is equivalent to $f_*(\Im\eta)\subseteq\Im\eta$ (the two subspaces have the same dimension). Therefore, any such $f$ must be a non-zero scalar multiple of $\operatorname{id}_{N}$. Consequently, $\Stab_{\Aut_B(N)}(\Im\eta)=\Bbbk^{\times}\operatorname{id}_{N}$. Since $\Bbbk^{\times}$ is isomorphic to $\mathbb{A}^1\setminus\{0\}$, we have $\dim\Stab_{\Aut_B(N)}(\Im\eta)=1$.

The automorphism group $\Aut_B(N)$ is the subset of $\End_B(N)$ consisting of those $B$-endomorphisms which are invertible as $B$-module endomorphisms. Hence, it is Zariski-open (see Example \ref{ex::open-set}). Since $\End_B(N)$ is an affine space of dimension $r(\mathbf{m})$, it is irreducible. Since $\Aut_B(N)$ is a nonempty Zariski-open subset of $\End_B(N)$, Proposition~\ref{prop:open-subset-dim} gives $\dim\Aut_B(N)=\dim\End_B(N)=r(\mathbf{m})$. Recall that the algebraic group $\Aut_B(N)$ acts on the Grassmannian $\Gr_{\ell}(\Hom_B(E,N))$ by $f\cdot L=f_*(L)$. For the subspace $\Im\eta$, the orbit $\Aut_B(N)\cdot \Im\eta$ is a locally closed subset by \cite[II. Proposition 8.3]{Humphreys-book-1975}. Then, the Orbit–Stabilizer Theorem (see, for example, \cite[Proposition 1.11]{Brion-actions-2010}) for algebraic group actions gives
\[
\dim\bigl(\Aut_B(N)\cdot\Im\eta\bigr)
= \dim\Aut_B(N)-\dim\Stab_{\Aut_B(N)}(\Im\eta)
= r(\mathbf{m})-1.
\]
Since the orbit is contained in $\Gr_{\ell}(\Hom_B(E,N))$, by Proposition \ref{prop::grassmannian-dimension}, we have
\[
r(\mathbf{m})-1 \leq \dim\Gr_{\ell}(\Hom_B(E,N))=\ell\bigl(h(\mathbf{m})-\ell\bigr).
\]
This yields the first inequality. The second inequality follows from the elementary bound $\ell(h-\ell)\leq h^2/4$ for any integer $h$.
\end{proof}

The dimension inequality in Theorem~\ref{thm:dim-conditions} records only the data $\ell$, $\mathbf{m}$, $h(\mathbf{m})$, and $r(\mathbf{m})$, but not the map $\eta\colon W\to\Hom_B(E,N)$, nor the position of $\Im\eta$ inside $\Hom_B(E,N)$. However, by Proposition~\ref{prop::mixed-brick-criterion}, the mixed brick condition depends precisely on these finer data: $\eta$ must be injective, and the stabilizer of $\Im\eta$ in $\End_B(N)$ must be exactly $\Bbbk\operatorname{id}_{N}$. This motivates us to separate the numerical obstruction from the geometric realization problem. We therefore introduce the following terminology.

\begin{definition}
Suppose $0\neq \mathbf{m} \in\mathbb{Z}_{\geq 0}^t$ and $1\leq\ell\leq h(\mathbf{m})$. A pair $(\mathbf{m},\ell)$ is called \emph{admissible} if $r(\mathbf{m})-1\leq \ell\bigl(h(\mathbf{m})-\ell\bigr)$. An admissible pair is called \emph{strict} if the inequality is strict, and \emph{critical} if equality holds. We say that $(\mathbf{m},\ell)$ is \emph{realized by a mixed brick}, or \emph{realizable}, if there exists a mixed brick $M=(N,W,\eta)$ with $N\cong N_{\mathbf{m}}$ and $\dim_{\Bbbk}W=\ell$.
\end{definition}

The following example verifies that admissibility is only a numerical condition: even for fixed $(\mathbf{m},\ell)$, different choices of $\eta$ may lead to different brick-theoretic behavior.
\begin{example}\label{ex::same-data-different-eta}
Let $B=\Bbbk$ and $E=\Bbbk^2$.
Then, $\Lambda=B[E]$ is the Kronecker algebra. There is only one indecomposable $B$-module up to isomorphism, namely $X=\Bbbk$.
Take $N=\Bbbk$ and $W=\Bbbk$. Then, $\ell=1$, $\mathbf{m}=(1)$, $h(\mathbf{m})=\dim_{\Bbbk}\Hom_{\Bbbk}(\Bbbk^2,\Bbbk)=2$ and $r(\mathbf{m})=1$, so that
\[
r(\mathbf{m})-1=0\leq1=\ell\bigl(h(\mathbf{m})-\ell\bigr).
\]
If $\eta_0\colon \Bbbk\to\Hom_B(\Bbbk^2,\Bbbk)$ is the zero map, then $\eta_0$ is not injective, so $M_0=(\Bbbk,\Bbbk,\eta_0)$ is not a mixed brick by Proposition~\ref{prop::mixed-brick-criterion}.
On the other hand, choose a nonzero map $p\colon \Bbbk^2\to\Bbbk$ and define $\eta_1\colon \Bbbk\to\Hom_B(\Bbbk^2,\Bbbk)$ by $\eta_1(1)=p$.
Then, $\eta_1$ is injective. Since $\End_B(\Bbbk)=\Bbbk\operatorname{id}_{N}$, the stabilizer condition in Proposition~\ref{prop::mixed-brick-criterion} holds automatically, and hence $M_1=(\Bbbk,\Bbbk,\eta_1)$ is a mixed brick. 
\end{example}

We then focus on the boundary case, where the above dimension inequality becomes an equality. The next two examples show that critical admissible pairs are restricted but not automatically realized.

\begin{example}\label{ex:critical-unrealized}
(A critical admissible pair may not be realized by a mixed brick.)
Let $B=\Bbbk[\varepsilon]/(\varepsilon^3)$ and $S$ be the simple $B$-module. Set $E=S^{\oplus 3}$ and take $\mathbf{m}=(\cdots, 0, 1,0,\cdots)$ with $N_{\mathbf{m}}=B$. Then, $r(\mathbf{m})=3$.
Moreover, $\Hom_B(S,N)\cong\soc N$ is one-dimensional, since $\soc N=\varepsilon^2B$.
Hence, $h(\mathbf{m})=3$ by $\Hom_B(E,N)\cong\Hom_B(S,N)^3$. For $\ell=1$, we have
\[
r(\mathbf{m})-1=2=\ell\bigl(h(\mathbf{m})-\ell\bigr).
\]
Thus, $(\mathbf{m},1)$ is a critical admissible pair. We claim that $(\mathbf{m},1)$ is not realized by any mixed brick. To see this, let $f\colon N\to N$ be the endomorphism given by multiplication by $\varepsilon$. Then, $f$ is nonzero and is not a scalar multiple of the identity. For every $\varphi\in\Hom_B(E,N)$, the image of $\varphi$ is contained in $\soc N=\varepsilon^2B$, and hence $f\circ\varphi=0$.
Namely, $f_*$ acts as zero on $\Hom_B(E,N)$.
By Lemma \ref{lem:zero-action}, no such module can be a mixed brick.
\end{example}

\begin{example}
(A critical admissible pair may be realized by a unique mixed brick.)
Let $B=E=N=W=\Bbbk$. Then, $h(\mathbf{m})=r(\mathbf{m})=\ell=1$, so $r(\mathbf{m})-1=0=\ell\bigl(h(\mathbf{m})-\ell\bigr)$. For any nonzero map $\eta\colon W\to\Hom_B(E,N)$, the module $M=(N,W,\eta)$ is a mixed brick. In this case, $\Gr_1(\Hom_B(E,N))$ consists of a single point, and therefore there is only one orbit and one isomorphism class.
\end{example}

At the end of this subsection, we mention the boundary case. If $(\mathbf{m},h(\mathbf{m}))$ is admissible, then $r(\mathbf{m})-1\leq0$, hence $r(\mathbf{m})=1$. Thus, $N_{\mathbf{m}}$ is a brick in $\mod B$, and the pair is critical.
Moreover, in this case, the only possible subspace is $L=\Hom_B(E,N_{\mathbf{m}})$, and $\Stab_{\End_B(N_{\mathbf{m}})}(L)=\End_B(N_{\mathbf{m}})$. Hence, $(\mathbf{m},h(\mathbf{m}))$ is realized by a mixed brick. In particular, strict admissible pairs can occur only for $1\leq\ell\leq h(\mathbf{m})-1$.

\subsection{Critical admissible pairs}
The two examples at the end of the last subsection suggest that critical admissible pairs exhibit a certain rigidity. Although such a pair need not be realized by a mixed brick, once it is realized, the corresponding mixed brick is expected to be unique up to isomorphism. This expectation is related to the critical equality, which suggests that the locus in the Grassmannian cut out by the mixed-brick criterion should contain at most one $\Aut_B(N_{\mathbf{m}})$-orbit. We will discuss this phenomenon more systematically in this subsection.

We continue to assume that $B$ is $E$-visible finite. For a nonzero vector $\mathbf{m}\in\mathbb{Z}_{\geq 0}^t$ and an integer $1\leq \ell\leq h(\mathbf{m})$, we define
\[
\mathcal{L}(\mathbf{m},\ell)
:=
\left\{
L\in \Gr_{\ell}\bigl(\Hom_B(E,N_{\mathbf{m}})\bigr)
\ \middle|\
\Stab_{\End_B(N_{\mathbf{m}})}(L)
=
\Bbbk\operatorname{id}_{N_{\mathbf{m}}}
\right\}.
\]
The group $\Aut_B(N_{\mathbf{m}})$ acts on $\mathcal{L}(\mathbf{m},\ell)$ by $f\cdot L:=f_*(L)$. Indeed, the stabilizer condition is preserved under the natural action of $\Aut_B(N_{\mathbf{m}})$ on $\Hom_B(E,N_{\mathbf{m}})$. The first point is that $\mathcal{L}(\mathbf{m},\ell)$ is not only set-theoretic, but is in fact an open subset of the Grassmannian.

\begin{lemma}\label{lem::Lml-open}
Let $0\neq \mathbf{m}\in\mathbb{Z}_{\geq 0}^t$ and $1\leq\ell\leq h(\mathbf{m})$. Then, the subset
\[
\mathcal{L}(\mathbf{m},\ell)
\subseteq
\Gr_{\ell}\bigl(\Hom_B(E,N_{\mathbf{m}})\bigr)
\]
is Zariski-open.
\end{lemma}
\begin{proof} 
We prove the assertion on each standard affine chart (see $\mathcal{L}_I$ in Proposition \ref{prop::grassmannian-dimension}), denoted by $\Omega$, of the Grassmannian $\Gr_\ell(\Hom_B(E,N_{\mathbf{m}}))$. Let $\Hom_B(E,N_{\mathbf{m}})=H'\oplus H''$ be a decomposition with $\dim_{\Bbbk} H'=\ell$. Then, each point $L\in\Omega$, viewed as an $\ell$-dimensional subspace of
$\Hom_B(E,N_{\mathbf{m}})$, can be uniquely written as $L=\Gamma(a):=\{u+a(u)\mid u\in H'\}$ for some linear map $a\colon H'\to H''$. For any $f\in\End_B(N_{\mathbf{m}})$, we write
\[ 
f_*= \begin{pmatrix} F_{11}&F_{12}\\ F_{21}&F_{22} \end{pmatrix}\colon  
\begin{pmatrix} H'\\ H'' \end{pmatrix}\longrightarrow 
\begin{pmatrix} H'\\ H'' \end{pmatrix},
\]
for the linear map induced by post-composition. Then, for any $u\in H'$, we have 
\[
f_*(u+a(u)) = \bigl(F_{11}(u)+F_{12}(a(u)),\, F_{21}(u)+F_{22}(a(u))\bigr), 
\]
and this vector belongs to $\Gamma(a)$ if and only if its second component is obtained by applying $a$ to its first component. More precisely, the condition $f_*(\Gamma(a))\subseteq \Gamma(a)$ is equivalent to $F_{21}+F_{22}\circ a-a\circ (F_{11}+F_{12}\circ a)=0$. Hence, for each $a\in\Hom_{\Bbbk}(H',H'')$, we obtain a linear map 
\[
\begin{aligned}
\Psi_a\colon \End_B(N_{\mathbf{m}}) &\longrightarrow \Hom_{\Bbbk}(H',H''),\\
f &\longmapsto F_{21}+F_{22}\circ a-a\circ (F_{11}+F_{12}\circ a).
\end{aligned}
\]
By our construction, $\Stab_{\End_B(N_{\mathbf{m}})}(\Gamma(a))=\ker\Psi_a$. Since the identity endomorphism $\operatorname{id}_{N_{\mathbf{m}}}$ preserves $\Gamma(a)$, $\operatorname{id}_{N_{\mathbf{m}}}\in\ker\Psi_a$ and $\dim_{\Bbbk}\ker\Psi_a\geq1$ for all $a$.
 
After choosing bases of $\End_B(N_{\mathbf{m}})$, $H'$ and $H''$, the matrix entries of $\Psi_a$ are polynomial functions in the coordinates of $a$. By the rank condition recorded in Example~\ref{ex::open-set}, the subset $\left\{ a\in \Hom_{\Bbbk}(H',H'') \ \middle|\ \dim_{\Bbbk}\ker\Psi_a\leq 1 \right\} $ is Zariski-open in the affine chart $\Omega$. In particular, the conditions $\dim_{\Bbbk}\ker\Psi_a\leq1$ and $\dim_{\Bbbk}\ker\Psi_a=1$ are equivalent. Moreover, $\Stab_{\End_B(N_{\mathbf{m}})}(L)=\Bbbk\operatorname{id}_{N_{\mathbf{m}}}$ if and only if its dimension is $1$. It follows that 
\[
\mathcal{L}(\mathbf{m},\ell)\cap\Omega
=
\left\{
\Gamma(a)\in\Omega
\ \middle|\
\dim_{\Bbbk}\ker\Psi_a=1
\right\}
=
\left\{
\Gamma(a)\in\Omega
\ \middle|\
\dim_{\Bbbk}\ker\Psi_a\leq1
\right\}
\]
is Zariski-open in $\Omega$. Since the standard affine charts cover $\Gr_\ell\bigl(\Hom_B(E,N_{\mathbf{m}})\bigr)$, we conclude that $\mathcal{L}(\mathbf{m},\ell)$ is Zariski-open in $\Gr_\ell(\Hom_B(E,N_{\mathbf{m}}))$. 
\end{proof}

By Proposition~\ref{prop::mixed-brick-criterion}, each subspace $L\in\mathcal{L}(\mathbf{m},\ell)$ gives a mixed brick $(N_{\mathbf{m}},L,\iota)$, where $\iota\colon L\hookrightarrow\Hom_B(E,N_{\mathbf{m}})$ is the natural inclusion. The remaining question is how to identify isomorphic mixed bricks among these subspaces; to answer this, we give a refined version of Proposition~\ref{prop::iso-orbit} as follows.

\begin{proposition}\label{prop::mixed-bricks-orbits} 
Let $0\neq \mathbf{m}\in\mathbb{Z}_{\geq 0}^t$ and $1\leq\ell\leq h(\mathbf{m})$. Then, the isomorphism classes of mixed bricks $M=(N,W,\eta)\in \mod \Lambda$ satisfying $N\cong N_{\mathbf{m}}$ and $\dim_{\Bbbk} W=\ell$ are in bijection with the $\Aut_B(N_{\mathbf{m}})$-orbits in $\mathcal{L}(\mathbf{m},\ell)$.
\end{proposition} 
\begin{proof} 
Let $M=(N,W,\eta)$ be a mixed brick such that $N\cong N_{\mathbf{m}}$ and $\dim_{\Bbbk} W=\ell$. By Lemma~\ref{lem:eta-injective}, the map $\eta$ is injective. Choose an isomorphism $\alpha\colon N\to N_{\mathbf{m}}$. Then, $\alpha$ induces a linear isomorphism $\alpha_*\colon \Hom_B(E,N)\to \Hom_B(E,N_{\mathbf{m}})$ by post-composition. Set $L:=\alpha_*(\Im\eta)\subseteq \Hom_B(E,N_{\mathbf{m}})$. Then, $L$ is an $\ell$-dimensional subspace. Since $M$ is a mixed brick, Proposition~\ref{prop::mixed-brick-criterion} gives $\Stab_{\End_B(N)}(\Im\eta)=\Bbbk\operatorname{id}_{N}$. Conjugating by the isomorphism $\alpha$, we obtain
\[
\Stab_{\End_B(N_{\mathbf{m}})}(L) = \alpha\,\Stab_{\End_B(N)}(\Im\eta)\,\alpha^{-1} = \Bbbk\operatorname{id}_{N_{\mathbf{m}}}.
\]
Hence, $L\in\mathcal{L}(\mathbf{m},\ell)$. If $\alpha'\colon N\to N_{\mathbf{m}}$ is another isomorphism, then $\alpha'=f\circ \alpha$ for some $f\in\Aut_B(N_{\mathbf{m}})$, and hence $\alpha'_*(\Im\eta)=f_*\alpha_*(\Im\eta)=f_*(L)$. It follows that every mixed brick of the prescribed type determines a well-defined $\Aut_B(N_{\mathbf{m}})$-orbit in $\mathcal{L}(\mathbf{m},\ell)$. 

Let $M'=(N',W',\eta')$ be another mixed brick of the prescribed type, and suppose that $(f,g)\colon M\to M'$ is an isomorphism. Then, $f\colon N\to N'$ is an isomorphism, and the compatibility condition gives $f_*(\Im\eta)=\Im\eta'$. If $\beta\colon N'\to N_{\mathbf{m}}$ is an isomorphism and $L':=\beta_*(\Im\eta')$, then $L'=\beta_*f_*(\Im\eta)=(\beta f\alpha^{-1})_*(L)$. Since $\beta f\alpha^{-1}\in\Aut_B(N_{\mathbf{m}})$, the subspaces $L$ and $L'$ lie in the same $\Aut_B(N_{\mathbf{m}})$-orbit.

Let $L\in\mathcal{L}(\mathbf{m},\ell)$ and $\iota\colon L\hookrightarrow \Hom_B(E,N_{\mathbf{m}})$ be the natural inclusion. Then,  Proposition~\ref{prop::mixed-brick-criterion} implies that $(N_{\mathbf{m}},L,\iota)$ is a mixed brick. It remains to identify when two subspaces give isomorphic mixed bricks. Let $L,L'\in\mathcal{L}(\mathbf{m},\ell)$, and consider the mixed bricks $(N_{\mathbf{m}},L,\iota)$ and $(N_{\mathbf{m}},L',\iota')$, where $\iota$ and $\iota'$ are the natural inclusions. By Proposition~\ref{prop::iso-orbit}, these two mixed bricks are isomorphic if and only if $L$ and $L'$ lie in the same $\Aut_B(N_{\mathbf{m}})$-orbit in $\Gr_\ell(\Hom_B(E,N_{\mathbf{m}}))$, which is equivalent to saying that they lie in the same $\Aut_B(N_{\mathbf{m}})$-orbit in $\mathcal{L}(\mathbf{m},\ell)$. 
\end{proof}

We combine the openness of $\mathcal{L}(\mathbf{m},\ell)$ with the orbit description to obtain the following result.

\begin{theorem}\label{thm::critical-pair-one-orbit}
Let $(\mathbf{m},\ell)$ be a critical admissible pair. Then,
$\mathcal{L}(\mathbf{m},\ell)$ is either empty or a single
$\Aut_B(N_{\mathbf{m}})$-orbit, and the mixed bricks realizing $(\mathbf{m},\ell)$ form at most one isomorphism class.
\end{theorem}
\begin{proof} 
By Proposition~\ref{prop::mixed-bricks-orbits}, it is enough to prove that $\mathcal{L}(\mathbf{m},\ell)$ contains at most one $\Aut_B(N_{\mathbf{m}})$-orbit. If $\mathcal{L}(\mathbf{m},\ell)=\varnothing$, there is nothing to prove. We then assume that $\mathcal{L}(\mathbf{m},\ell)\neq\varnothing$. By Lemma~\ref{lem::Lml-open}, $\mathcal{L}(\mathbf{m},\ell)$ is a nonempty Zariski-open subset of $\Gr_\ell(\Hom_B(E,N_{\mathbf{m}}))$. Since the Grassmannian is irreducible, $\mathcal{L}(\mathbf{m},\ell)$ is also irreducible. Moreover, by Propositions~\ref{prop:open-subset-dim} and \ref{prop::grassmannian-dimension}, we have $\dim \mathcal{L}(\mathbf{m},\ell) = \dim \Gr_\ell(\Hom_B(E,N_{\mathbf{m}})) = \ell\bigl(h(\mathbf{m})-\ell\bigr)$. Let $L\in\mathcal{L}(\mathbf{m},\ell)$. Similar to the proof of Theorem \ref{thm:dim-conditions}, the orbit $\Aut_B(N_{\mathbf{m}})\cdot L$ is locally closed in $\Gr_\ell(\Hom_B(E,N_{\mathbf{m}}))$ and 
\[ 
\dim\bigl(\Aut_B(N_{\mathbf{m}})\cdot L\bigr) = \dim\Aut_B(N_{\mathbf{m}}) - \dim\Stab_{\Aut_B(N_{\mathbf{m}})}(L) = r(\mathbf{m})-1. 
\] 
Since $(\mathbf{m},\ell)$ is critical, we have $r(\mathbf{m})-1=\ell(h(\mathbf{m})-\ell)$, and  $\dim\bigl(\Aut_B(N_{\mathbf{m}})\cdot L\bigr) = \dim\mathcal{L}(\mathbf{m},\ell)$. Since the orbit is open in its closure, $\Aut_B(N_{\mathbf{m}})\cdot L$ is a nonempty Zariski-open subset of $\mathcal{L}(\mathbf{m},\ell)$.

Suppose that $\mathcal{L}(\mathbf{m},\ell)$ contains two distinct $\Aut_B(N_{\mathbf{m}})$-orbits. By the previous paragraph, both orbits are nonempty Zariski-open subsets of the irreducible space $\mathcal{L}(\mathbf{m},\ell)$. Hence, they must intersect. This contradicts the fact that distinct orbits are disjoint. Thus, $\mathcal{L}(\mathbf{m},\ell)$ contains at most one $\Aut_B(N_{\mathbf{m}})$-orbit.
\end{proof}

\subsection{Strict admissible pairs} 
We now turn to strict admissible pairs. In the critical case, the dimension count is exactly balanced, which leads to the uniqueness result for realized critical pairs. The strict case is different: the strict inequality creates a positive numerical gap, and one may expect this gap to be related to infinite phenomena. The goal of this subsection is to make this expectation precise at the numerical level. We show that the existence of one strict admissible pair forces the existence of infinitely many (strict) admissible pairs. Conversely, if there are infinitely many admissible pairs, then strict admissible pairs must occur. 

Before stating the result precisely, we extend $h$ and $r$ from $\mathbb{Z}_{\geq 0}^t$ to $\mathbb{R}_{\geq 0}^t$. For any $\mathbf{x}\in \mathbb{R}_{\geq 0}^t$, we define $h(\mathbf{x}):=\mathbf{d}\cdot\mathbf{x}$ and $r(\mathbf{x}):=\mathbf{x}^{\mathsf{T}}\mathsf{G}\mathbf{x}$, where
\[
\mathbf{d}:=(d_1,d_2,\ldots,d_t) \; \in \mathbb{Z}_{\geq 0}^t
\qquad\text{and}\qquad
\mathsf{G}:=\frac{\mathsf{C}+\mathsf{C}^{\mathsf{T}}}{2}.
\] 
We then define a quadratic form on $\mathbb{R}_{\geq 0}^t$ by $Q(\mathbf{x}):=4r(\mathbf{x})-h(\mathbf{x})^2$. Since $\mathbf{x}^{\mathsf{T}}\mathsf{C}\mathbf{x}
=
\mathbf{x}^{\mathsf{T}}\mathsf{G}\mathbf{x}$ for every
$\mathbf{x}\in\mathbb{R}^t$, this definition agrees with the previous one on $\mathbb{Z}_{\geq 0}^t$. In particular, for any $\mathbf{m}\in\mathbb{Z}_{\geq 0}^t$, we have $Q(\mathbf{m})=4r(\mathbf{m})-h(\mathbf{m})^2$. 

The next lemma gives the basic positivity consequence of the absence of strict admissible pairs. It is the key estimate that allows us to pass from the discrete set of multiplicity vectors to the positive cone.

\begin{lemma}\label{lem::no-strict-positive-Q}
Suppose that $B$ is $E$-visible finite and there is no strict admissible pair. Then, $Q(\mathbf{x})>0$ for every nonzero vector $\mathbf{x}\in\mathbb{R}_{\geq 0}^t$. 
\end{lemma}
\begin{proof}
We first show that $Q(\mathbf{m})>0$ for every nonzero $\mathbf{m}\in\mathbb{Z}_{\geq 0}^t$. If $Q(\mathbf{m})<0$, then $4r(\mathbf{m})<h(\mathbf{m})^2$. There exists an integer $\ell$ with $1\leq \ell\leq h(\mathbf{m})$ such that
\[
r(\mathbf{m})-1
<
\left\lfloor \frac{h(\mathbf{m})^2}{4}\right\rfloor
=
\ell\bigl(h(\mathbf{m})-\ell\bigr),
\]
and then $(\mathbf{m},\ell)$ is a strict admissible pair, a contradiction. If $Q(\mathbf{m})=0$, then $4r(\mathbf{m})=h(\mathbf{m})^2$. Since $r(\mathbf{m})$ is an integer, $h(\mathbf{m})$ is even. Taking $\ell=h(\mathbf{m})/2$, we obtain
\[
r(\mathbf{m})-1
=
\frac{h(\mathbf{m})^2}{4}-1
<
\frac{h(\mathbf{m})^2}{4}
=
\ell\bigl(h(\mathbf{m})-\ell\bigr),
\]
such that $(\mathbf{m},\ell)$ is again a strict admissible pair, a contradiction.

Let $0\neq \mathbf{q}\in\mathbb{Q}_{\geq 0}^t$. Choose a positive integer
$s$ such that $s\mathbf{q}\in\mathbb{Z}_{\geq 0}^t$. Since $s\mathbf{q}\neq0$, the previous paragraph gives $Q(s\mathbf{q})>0$. On the other hand, $Q$ is homogeneous of degree $2$, so $Q(s\mathbf{q})=s^2Q(\mathbf{q})$. Hence, $Q(\mathbf{q})>0$ for every nonzero $\mathbf{q}\in\mathbb{Q}_{\geq 0}^t$. As we mentioned in Subsection \ref{subsec::Zariski-topology}, $\mathbb{Q}^t$ is dense in $\mathbb{R}^t$. For every $\mathbf{x}\in\mathbb{R}_{\geq 0}^t$, we can choose a sequence
$\mathbf{q}_n\in\mathbb{Q}_{\geq 0}^t$ such that $\mathbf{q}_n\to\mathbf{x}$ in the usual Euclidean topology. Since $Q$ is a polynomial function, it is continuous. Hence, $Q(\mathbf{x})=\lim_{n\to\infty}Q(\mathbf{q}_n)\geq0$. 

Suppose that $Q(\mathbf{x})=0$ for some nonzero $\mathbf{x}\in\mathbb{R}_{\geq 0}^t$. Let $I:=\{i\mid x_i>0\}$ be the support of $\mathbf{x}$. Then, $x_j=0$ for every $j\notin I$. If $\mathbf{u}$ is a vector supported on $I$, then there exists $\delta>0$ such that $\mathbf{x}+s\mathbf{u}$ belongs to $\mathbb{R}_{\geq 0}^t$ for every real number $s$ with $|s|<\delta$. Since we have already shown that $Q$ is nonnegative on $\mathbb{R}_{\geq 0}^t$ and $Q(\mathbf{x})=0$, the one-variable polynomial $f(s):=Q(\mathbf{x}+s\mathbf{u})$ satisfies $f(s)\geq0$ for all real numbers $s$ with $|s|<\delta$, and $f(0)=0$. Thus, $s=0$ is a local minimum of $f$, and hence $f'(0)=0$.

Set $\mathsf{A}:=4\mathsf{G}-\mathbf{d}\mathbf{d}^{\mathsf{T}}$. Then, $\mathsf{A}$ is the symmetric matrix with rational entries such that $Q(\mathbf{z})=\mathbf{z}^{\mathsf{T}}\mathsf{A}\mathbf{z}$, for any $\mathbf{z}\in\mathbb{R}^t$. We have
\[
f(s)=Q(\mathbf{x}+s\mathbf{u})
=
Q(\mathbf{x})+2s\,\mathbf{u}^{\mathsf{T}}\mathsf{A}\mathbf{x}+s^2Q(\mathbf{u}).
\]
Therefore, $f'(0)=2\mathbf{u}^{\mathsf{T}}\mathsf{A}\mathbf{x}$. Since $f'(0)=0$, we obtain $\mathbf{u}^{\mathsf{T}}\mathsf{A}\mathbf{x}=0$. Let $\mathsf{A}_I$ be the submatrix of $\mathsf{A}$ whose rows and columns are indexed by $I$. Since both $\mathbf{x}$ and $\mathbf{u}$ are supported on $I$, $\mathbf{u}_I^{\mathsf{T}}\mathsf{A}_I\mathbf{x}_I=0$, for every $\mathbf{u}_I\in\mathbb{R}^I$. Taking $\mathbf{u}_I=\mathsf{A}_I\mathbf{x}_I$, we obtain $(\mathsf{A}_I\mathbf{x}_I)^{\mathsf{T}}(\mathsf{A}_I\mathbf{x}_I)=0$, and hence $\mathsf{A}_I\mathbf{x}_I=0$.

Since $\mathsf{A}_I$ has rational entries, the solution space of the linear system $\mathsf{A}_I\mathbf{z}=0$ has a basis consisting of rational vectors, say $\mathbf{v}_1,\mathbf{v}_2,\ldots,\mathbf{v}_m$. Since $\mathbf{x}_I\in\ker\mathsf{A}_I$, we can write $\mathbf{x}_I=a_1\mathbf{v}_1+\cdots+a_m\mathbf{v}_m$ with $a_1,\ldots,a_m\in\mathbb{R}$. Choose rational numbers $b_1,\ldots,b_m$ sufficiently close to $a_1,\ldots,a_m$, and set $\mathbf{y}_I=b_1\mathbf{v}_1+\cdots+b_m\mathbf{v}_m$. Then, $\mathbf{y}_I\in\ker\mathsf{A}_I\cap\mathbb{Q}^I$ and $\mathbf{y}_I$ is sufficiently close to $\mathbf{x}_I$. Since all entries of $\mathbf{x}_I$ are strictly positive, we may choose the $b_j$ so close to the $a_j$ that all entries of $\mathbf{y}_I$ are still strictly positive. Thus, $\mathbf{y}_I\in\mathbb{Q}_{>0}^I$. Extending $\mathbf{y}_I$ by zero outside $I$, we obtain a nonzero vector $\mathbf{y}\in\mathbb{Q}_{\geq 0}^t$. Since $\mathsf{A}_I\mathbf{y}_I=0$, we have $Q(\mathbf{y})=\mathbf{y}^{\mathsf{T}}\mathsf{A}\mathbf{y}
=
\mathbf{y}_I^{\mathsf{T}}\mathsf{A}_I\mathbf{y}_I
=0$. This contradicts the positivity of $Q$ on nonzero rational vectors. Therefore, $Q(\mathbf{x})>0$ for every nonzero $\mathbf{x}\in\mathbb{R}_{\geq 0}^t$. 
\end{proof}

The above lemma gives one direction of the relation between strictness and positivity of $Q(\mathbf{x})$. The following theorem shows that this positivity condition is in fact equivalent to the finiteness of admissible pairs.

\begin{theorem}\label{thm:strict-pairs}
Suppose that $B$ is $E$-visible finite. The following statements are equivalent.
\begin{enumerate}
\item There exists no strict admissible pair.
\item There are finitely many admissible pairs. 
\item $Q(\mathbf{x})>0$ for every nonzero vector $\mathbf{x}\in\mathbb{R}_{\geq 0}^t$. 
\item $Q(\mathbf{m})>0$ for every nonzero vector $\mathbf{m}\in\mathbb{Z}_{\geq 0}^t$. 
\end{enumerate}
\end{theorem} 
\begin{proof} 
We first prove that (2) implies (1). Suppose that $(\mathbf{m},\ell)$ is a strict admissible pair. We have $r(\mathbf{m})\leq\ell(h(\mathbf{m})-\ell)$ since both sides are integers. For every positive integer $z$, set $\mathbf{n}=z\mathbf{m}$ and $q=z\ell$. Then, $1\leq q\leq h(z\mathbf{m})$, $h(z\mathbf{m})=z h(\mathbf{m})$ and $r(z\mathbf{m})=z^2 r(\mathbf{m})$. Therefore,
\[
r(z\mathbf{m})-1 = z^2r(\mathbf{m})-1 < z^2\ell\bigl(h(\mathbf{m})-\ell\bigr) = q\bigl(h(z\mathbf{m})-q\bigr).
\]
It follows that $(z\mathbf{m},z\ell)$ is a strict admissible pair for every $z\geq1$, and these pairs are pairwise distinct.
Thus, there are infinitely many admissible pairs. 

We next prove that (1) implies (2). Suppose that there is no strict admissible pair. Recall $\mathbf{d}=(d_1,d_2,\ldots,d_t) \; \in \mathbb{Z}_{\geq 0}^t$. We consider $\Delta:=\left\{
\mathbf{x}\in\mathbb{R}_{\geq 0}^t
\ \middle|\
\mathbf{d}\cdot \mathbf{x}=1
\right\}$. Since $d_i>0$, every $\mathbf{x}\in\Delta$ satisfies $0\leq x_i\leq 1/d_i$. Both the condition $\mathbf{d}\cdot\mathbf{x}=1$ and $x_i\geq0$ define closed subsets of $\mathbb{R}^t$. Hence, $\Delta$ is closed and bounded in $\mathbb{R}^t$, and therefore compact by the Heine--Borel theorem.
By Lemma~\ref{lem::no-strict-positive-Q}, $Q(\mathbf{x})>0$ for any $\mathbf{x}\in\Delta$. Since $Q$ is a polynomial function, it is continuous on $\Delta$. Then, $Q$ has a positive minimum on $\Delta$ by the Extreme Value Theorem. There exists a constant $\delta>0$ such that $Q(\mathbf{x})\geq\delta$ for all $\mathbf{x}\in\Delta$.

Let $0\neq\mathbf{m}\in\mathbb{Z}_{\geq 0}^t$. We have $h(\mathbf{m})=\mathbf{d}\cdot \mathbf{m}>0$, and therefore $\mathbf{m}/h(\mathbf{m})\in\Delta$. Using the homogeneity of $Q$, we obtain
\[
Q(\mathbf{m}) = h(\mathbf{m})^2 Q\left(\frac{\mathbf{m}}{h(\mathbf{m})}\right)
\geq \delta h(\mathbf{m})^2.
\]
Let $(\mathbf{m},\ell)$ be an admissible pair. It follows that $Q(\mathbf{m}) = 4r(\mathbf{m})-h(\mathbf{m})^2 \leq 4$. Combining this with the lower bound above gives $\delta h(\mathbf{m})^2\leq 4$. Hence, $h(\mathbf{m})$ is bounded and all $m_i$ are bounded, so only finitely many vectors $\mathbf{m}$ can occur. For each fixed $\mathbf{m}$, there are only finitely many possible integers $\ell$ satisfying $1\leq\ell\leq h(\mathbf{m})$. Therefore, there are only finitely many admissible pairs.

By Lemma~\ref{lem::no-strict-positive-Q}, (1) implies (3). It remains to prove that (3) implies (1). Suppose that $Q(\mathbf{x})>0$ for every nonzero vector $\mathbf{x}\in\mathbb{R}_{\geq 0}^t$. If there is a strict admissible pair $(\mathbf{m},\ell)$, then
\[
r(\mathbf{m})
\leq
\ell\bigl(h(\mathbf{m})-\ell\bigr)
\leq
\frac{h(\mathbf{m})^2}{4}.
\]
Hence $Q(\mathbf{m}) = 4r(\mathbf{m})-h(\mathbf{m})^2 \leq 0$, which contradicts the positivity of $Q(\mathbf{x})$ on $\mathbb{R}_{\geq 0}^t\setminus\{0\}$.
Therefore, there is no strict admissible pair.

Lastly, (4) $\Rightarrow$ (3) is given in the proof of Lemma~\ref{lem::no-strict-positive-Q}, and the converse is obvious.
\end{proof}

We emphasize that Theorem~\ref{thm:strict-pairs} is purely numerical: the infinitely many strict admissible pairs obtained there may all fail to be realized by mixed bricks. 

\begin{example}
Let $B=\Bbbk[\varepsilon]/(\varepsilon^4)$ and $S$ be the simple $B$-module. Set $E=S^{\oplus 4}$ and take $\mathbf{m}=(\cdots, 0, 1,0,\cdots)$ with $N_{\mathbf{m}}=B$. Then, $r(\mathbf{m})=4$. Moreover, $\Hom_B(S,N)\cong\soc N=\varepsilon^3B$ is one-dimensional. Hence, $h(\mathbf{m})=4$ by $\Hom_B(E,N)\cong\Hom_B(S,N)^4$. Taking $\ell=2$, we obtain $r(\mathbf{m})-1=3<4=\ell\bigl(h(\mathbf{m})-\ell\bigr)$. Thus, $(\mathbf{m},2)$ is a strict admissible pair. Similar to Example \ref{ex:critical-unrealized}, $(\mathbf{m},2)$ is not realized by any mixed brick. By Theorem~\ref{thm:strict-pairs}, the admissible pairs $(z\mathbf{m},2z), z\geq1$ are all strict, but no module $M=(B^{\oplus z},W,\eta)$ with $\dim_{\Bbbk}W=2z$ can be a mixed brick. 
\end{example}

However, once a strict admissible pair is realized by a mixed brick, we expect the strict inequality to lead to many more mixed bricks. We will show this phenomenon in the next section. Here, we illustrate the phenomenon in the simplest example. 

\begin{example}
(A strict admissible pair may lead to infinitely many mixed bricks.) Let $B=\Bbbk$, $E=\Bbbk^2$, $N=\Bbbk$, and $W=\Bbbk$ as in Example~\ref{ex::same-data-different-eta}. For any one-dimensional subspace $L\in\Gr_1(\Hom_B(E,N))$, we choose a nonzero element $\varphi\in L$ and define $\eta\colon W\to\Hom_B(E,N)$ by $\eta(1)=\varphi$. Then, $\eta$ is injective and $\Im\eta=\Bbbk\varphi=L$. Since $\End_B(N)=\Bbbk\operatorname{id}_{N}$, each such $\eta$ gives a mixed brick $M=(N,W,\eta)$ by Proposition \ref{prop::mixed-brick-criterion}.
Moreover, $\Gr_1(\Hom_B(E,N))\cong\mathbb{P}^1$ has infinitely many one-dimensional subspaces, and the group $\Aut_B(N)=\Bbbk^{\times}$ acts on $\Hom_B(E,N)$ by scalar multiplication, so it fixes every one-dimensional subspace in $\Gr_1(\Hom_B(E,N))$. By Proposition~\ref{prop::iso-orbit}, the strict admissible pair $(\mathbf{m},\ell)=((1),1)$ gives infinitely many isomorphism classes of mixed bricks.
\end{example}

\subsection{A tangent-map interpretation}
We give a tangent-map interpretation of the stabilizer condition appearing in Proposition~\ref{prop::mixed-brick-criterion}. This viewpoint explains why the numerical inequality in Theorem~\ref{thm:dim-conditions} is the natural dimension condition, and it also clarifies the difference between critical and strict admissible pairs.

Fix a nonzero multiplicity vector $\mathbf{m}\in\mathbb{Z}_{\geq 0}^t$ and an integer $1\leq \ell\leq h(\mathbf{m})$. Let $L\in\Gr_\ell(\Hom_B(E,N_{\mathbf{m}}))$. For every $f\in\End_B(N_{\mathbf{m}})$, post-composition induces a linear endomorphism $f_*$ of $\Hom_B(E,N_{\mathbf{m}})$, as we mentioned in Definition \ref{def::E-faithful-module}. We define a $\Bbbk$-linear map
\[
\begin{aligned}
\partial_L\colon
\End_B(N_{\mathbf{m}})
&\longrightarrow
\Hom_{\Bbbk}\bigl(L,\Hom_B(E,N_{\mathbf{m}})/L\bigr)\\
f & \longmapsto\partial_L(f)(\varphi):=f_*(\varphi)+L \quad (\forall \; \varphi\in L).
\end{aligned}
\]
By definition, an endomorphism $f\in\End_B(N_{\mathbf{m}})$ belongs to $\ker\partial_L$ if and only if $f_*(L)\subseteq L$. Hence, $\ker\partial_L = \Stab_{\End_B(N_{\mathbf{m}})}(L)$. In particular, Proposition~\ref{prop::mixed-brick-criterion} says that the triple $(N_{\mathbf{m}},L,\iota)$, where $\iota\colon L\hookrightarrow\Hom_B(E,N_{\mathbf{m}})$ is the natural inclusion, is a mixed brick if and only if $\ker\partial_L=\Bbbk\operatorname{id}_{N_{\mathbf{m}}}$. 

Since every scalar multiple of $\operatorname{id}_{N_{\mathbf{m}}}$ preserves $L$, the map $\partial_L$ induces a linear map
\[
\overline{\partial}_L\colon
\End_B(N_{\mathbf{m}})/\Bbbk\operatorname{id}_{N_{\mathbf{m}}}
\longrightarrow
\Hom_{\Bbbk}\bigl(L,\Hom_B(E,N_{\mathbf{m}})/L\bigr)
\]
defined by $\overline{\partial}_L \left( f+\Bbbk\operatorname{id}_{N_{\mathbf{m}}} \right) (\varphi) = f_*(\varphi)+L$ for any $\varphi\in L$. It follows that $L\in\mathcal{L}(\mathbf{m},\ell)$ if and only if $\overline{\partial}_L$ is injective, or equivalently, $\ker\partial_L=\Bbbk\operatorname{id}_{N_{\mathbf{m}}}$. On the other hand, $\dim_{\Bbbk}\End_B(N_{\mathbf{m}})/\Bbbk\operatorname{id}_{N_{\mathbf{m}}}
= r(\mathbf{m})-1$ and $\dim_{\Bbbk}\Hom_{\Bbbk}\bigl(L,\Hom_B(E,N_{\mathbf{m}})/L\bigr)
=
\ell\bigl(h(\mathbf{m})-\ell\bigr)$ imply that the injectivity of $\overline{\partial}_L$ can occur only if $r(\mathbf{m})-1\leq \ell\bigl(h(\mathbf{m})-\ell\bigr)$. Thus, admissibility is precisely the necessary dimension condition for
$\overline{\partial}_L$ to be injective.

If $(\mathbf{m},\ell)$ is critical and $L\in\mathcal{L}(\mathbf{m},\ell)$, then $\overline{\partial}_L$ is an isomorphism, so the orbit of $L$ has full dimension in the Grassmannian. This is the reason behind Theorem~\ref{thm::critical-pair-one-orbit}. If $(\mathbf{m},\ell)$ is strict and $L\in\mathcal{L}(\mathbf{m},\ell)$, then $\overline{\partial}_L$ is injective but not surjective, and the orbit of $L$ has strictly smaller dimension than $\mathcal{L}(\mathbf{m},\ell)$. This is precisely the dimension gap used in Proposition~\ref{prop:strict-infinite-bricks}, in the next section.

\section{Brick-finiteness}\label{Section-4}
We now return from the numerical problem to the representation-theoretic problem. 
Theorem~\ref{thm:strict-pairs} identifies strict admissible pairs as the numerical source of infiniteness. To obtain actual mixed bricks, one must also realize the corresponding subspaces with a scalar stabilizer. Theorem~\ref{thm:strict-pairs} also says that the absence of strict admissible pairs is equivalent to the finiteness of all admissible pairs. In this situation, every admissible pair is critical, and the uniqueness theorem for realized critical pairs gives finiteness of mixed bricks.

\begin{proposition}\label{prop::mixed-bricks-finite}
Suppose that $B$ is $E$-visible finite and $\Lambda=B[E]$. If there is no strict admissible pair, then there are only finitely many isomorphism classes of mixed bricks in $\mod \Lambda$. If moreover, $B$ is brick-finite, then $\Lambda$ is brick-finite.
\end{proposition}
\begin{proof}
If there is no admissible pair, then there is no mixed brick in $\mod \Lambda$ by Theorem~\ref{thm:dim-conditions}. If there exist admissible pairs, by Theorem~\ref{thm:strict-pairs}, the absence of strict admissible
pairs implies that there are only finitely many admissible pairs,
and all of them are critical. By
Theorem~\ref{thm::critical-pair-one-orbit}, each critical admissible
pair is realized by at most one isomorphism class of mixed bricks.
It follows that there are only finitely many isomorphism classes of
mixed bricks in $\mod\Lambda$.

Let $M=(N,W,\eta)$ be a brick in $\mod \Lambda$. If $M$ is mixed, then $M$ belongs to one of the finitely many isomorphism classes considered above. If $W=0$, then $M=(N,0,0)$ is a brick in $\mod \Lambda$ if and only if $N$ is a brick in $\mod B$. Since $B$ is brick-finite, there are only finitely many such isomorphism classes. If $N=0$, then $M=(0,W,0)$ is a brick if and only if $\dim_{\Bbbk} W=1$; this case contributes only the single isomorphism class $(0,\Bbbk,0)$. Therefore, every brick in $\mod \Lambda$ is either a mixed brick, a brick of the form $(N,0,0)$ coming from $\mod B$, or the simple module $(0,\Bbbk,0)$ at the extension vertex. Consequently, $\Lambda$ is brick-finite.
\end{proof}

We explain how Proposition~\ref{prop::mixed-bricks-finite} can be used in practice. For a fixed nonzero vector $\mathbf{m}\in\mathbb{Z}_{\geq 0}^t$, the maximum of $\ell(h(\mathbf{m})-\ell)$ over all integers $1\leq\ell\leq h(\mathbf{m})$ is $\left\lfloor h(\mathbf{m})^2/4\right\rfloor$. Hence, there exists an admissible pair with first component $\mathbf{m}$ if and only if $r(\mathbf{m})-1\leq \left\lfloor h(\mathbf{m})^2/4\right\rfloor$. It follows that the finiteness of admissible pairs is equivalent to the finiteness of the set 
\[
\mathcal{S}:=\{\mathbf{m}\in\mathbb{Z}_{\geq 0}^t\setminus\{0\}\mid r(\mathbf{m})-1\leq \left\lfloor h(\mathbf{m})^2/4\right\rfloor\}.
\]
This formulation is often more convenient for computations.

\begin{remark}
Recall $\Delta=\{\mathbf{x}\in\mathbb{R}_{\geq 0}^t\mid h(\mathbf{x})=1\}$ in the proof of Theorem~\ref{thm:strict-pairs}. Since $Q(\mathbf{x})=4r(\mathbf{x})-1$ for $\mathbf{x}\in\Delta$, $\mathcal{S}$ is finite if and only if $\min_{\mathbf{x}\in\Delta} r(\mathbf{x})>1/4$. In practice, this reduces the problem to a finite-dimensional quadratic minimization problem on the simplex $\Delta$.
One may compute the minimum by checking the faces of $\Delta$, or equivalently by fixing the support $I\subseteq\{1,\ldots,t\}$ and minimizing $\mathbf{x}_I^{\mathsf{T}}\mathsf{G}_I\mathbf{x}_I$ under the constraints $\mathbf{x}_I\geq0$ and $\mathbf{d}_I\cdot\mathbf{x}_I=1$.
A quicker sufficient test is to prove a lower bound $r(\mathbf{x})\geq c\,h(\mathbf{x})^2$ on $\mathbb{R}_{\geq 0}^t$ for some $c>1/4$; this will be used in Subsection \ref{subsec-example}. Conversely, to prove that $\mathcal{S}$ is infinite, it is enough to find a nonzero vector $\mathbf{x}\in\mathbb{R}_{\geq 0}^t$ with $Q(\mathbf{x})\leq0$, or more concretely a nonzero integer vector $\mathbf{m}\in\mathbb{Z}_{\geq 0}^t$ with $4r(\mathbf{m})\leq h(\mathbf{m})^2$.
\end{remark}
 
The next proposition shows that once a strict admissible pair is realized by a mixed brick, the strict dimension gap produces infinitely many isomorphism classes of mixed bricks.

\begin{proposition}\label{prop:strict-infinite-bricks}
Let $(\mathbf{m},\ell)$ be a strict admissible pair. If $(\mathbf{m},\ell)$ is realized by a mixed brick, then there are infinitely many isomorphism classes of mixed bricks in $\mod \Lambda$. In particular, $\Lambda$ is brick-infinite.
\end{proposition}
\begin{proof}
Since $(\mathbf{m},\ell)$ is realized by a mixed brick, Proposition~\ref{prop::mixed-bricks-orbits} implies that $\mathcal{L}(\mathbf{m},\ell)$ is nonempty. It is shown in the proof of Theorem \ref{thm::critical-pair-one-orbit} that $\dim \mathcal{L}(\mathbf{m},\ell)=\ell\bigl(h(\mathbf{m})-\ell\bigr)$. We have also shown in the proof of Theorem \ref{thm:dim-conditions} that every $\Aut_B(N_{\mathbf{m}})$-orbit in $\mathcal{L}(\mathbf{m},\ell)$ has dimension $r(\mathbf{m})-1$. Since $(\mathbf{m},\ell)$ is strict, we have $r(\mathbf{m})-1<\ell\bigl(h(\mathbf{m})-\ell\bigr)=\dim\mathcal{L}(\mathbf{m},\ell)$.

Suppose, for a contradiction, that $\mathcal{L}(\mathbf{m},\ell)$ contains only finitely many $\Aut_B(N_{\mathbf{m}})$-orbits. Since these orbits are locally closed (see the proof of Theorem \ref{thm:dim-conditions}), the closure of each orbit in $\mathcal{L}(\mathbf{m},\ell)$ has dimension $r(\mathbf{m})-1$. Hence, the finite union of these orbit closures has dimension $r(\mathbf{m})-1$. But this finite union covers $\mathcal{L}(\mathbf{m},\ell)$, contradicting $\dim\mathcal{L}(\mathbf{m},\ell)>r(\mathbf{m})-1$. Thus, $\mathcal{L}(\mathbf{m},\ell)$ contains infinitely many $\Aut_B(N_{\mathbf{m}})$-orbits.

Lastly, Proposition~\ref{prop::mixed-bricks-orbits} identifies the $\Aut_B(N_{\mathbf{m}})$-orbits in $\mathcal{L}(\mathbf{m},\ell)$ with the isomorphism classes of mixed bricks $M=(N',W,\eta)$ satisfying $N'\cong N_{\mathbf{m}}$ and $\dim_{\Bbbk} W=\ell$. It follows that there are infinitely many isomorphism classes of mixed bricks in $\mod \Lambda$.
\end{proof}

We point out that the existence of a strict admissible pair is not a sufficient condition for the existence of a realized strict admissible pair. There are various reasons why a strict admissible pair $(\mathbf{m},\ell)$ fails to be realized by a mixed brick. We give an example using the reason that $N_{\mathbf{m}}$ is not $E$-faithful. 

\subsection{Numerical strictness without geometric realization}\label{subsec-example}
In this subsection, we set $B=\Bbbk Q_n$ with 
\[
Q_n\colon \xymatrix@C=1cm{1\ar[r]&2\ar[r]&\cdots \ar[r]&n}, \quad n\geq 2.
\]
For $1\leq a\leq b\leq n$, let $M[a,b]$ be the interval module which is one-dimensional at the vertices $a,a+1,\ldots,b$, with identity maps along the arrows inside the interval and zero elsewhere. The indecomposable $B$-modules are precisely these interval modules. With our convention for right modules, the following interval Hom formula holds:
\[
\dim_{\Bbbk}\Hom_B(M[a,b],M[c,d])=
\begin{cases}
1, & c\leq a\leq d\leq b,\\
0, & \text{otherwise.}
\end{cases}
\]
It follows that for a single interval $E_0=M[a,b]$, the $E_0$-visible indecomposable modules are precisely the intervals $M[c,d]$ with $c\leq a\leq d\leq b$. Equivalently, they form the rectangle
$\{M[c,d]\mid 1\leq c\leq a,\ a\leq d\leq b\}$.

We consider $n=6$ as an example. Take $E=M[1,1]\oplus M[2,2]\oplus M[3,6]$ and set 
\[ 
N= M[1,1]\oplus M[1,6]\oplus M[2,2]\oplus M[2,5]\oplus M[3,3]\oplus M[3,4]. 
\] 
One may easily verify that $M[1,1]$ is visible from $M[1,1]$, $M[2,2]$ is visible from $M[2,2]$, and the remaining four summands are visible from $M[3,6]$. By the interval Hom formula, the only nonzero Hom-spaces between two distinct indecomposable summands of $N$ are 
\[ 
\Hom_B(M[1,6],M[1,1]),\qquad \Hom_B(M[2,5],M[2,2]),\qquad \Hom_B(M[3,4],M[3,3]). 
\] 
Hence, $\dim_{\Bbbk}\End_B(N)=9$. On the other hand, each indecomposable summand of $N$ contributes exactly one dimension to $\Hom_B(E,N)$, and therefore $\dim_{\Bbbk}\Hom_B(E,N)=6$. Taking $\ell=3$, we obtain 
\[ 
\dim_{\Bbbk}\End_B(N)-1=8<9=\ell\bigl(\dim_{\Bbbk}\Hom_B(E,N)-\ell\bigr). 
\] 
It turns out that the pair $(\mathbf{m},3)$ with $N_{\mathbf{m}}\cong N$ is a strict admissible pair. 

We first claim that this strict admissible pair $(\mathbf{m},3)$ is not realized by a mixed brick. Let $0\neq f\colon M[1,6]\to M[1,1]$ be the unique nonzero morphism, and regard $f$ as an endomorphism of $N$ whose only nonzero matrix component is this map. Then, $f$ is nonzero and non-scalar in $\End_B(N)$. We find that $f$ acts as zero on $\Hom_B(E,N)$ by post-composition. Indeed, the only possible nonzero contribution would come from composing a morphism $M[3,6]\to M[1,6]$ with $f$, but the resulting morphism lies in $\Hom_B(M[3,6],M[1,1])=0$; the summands $M[1,1]$ and $M[2,2]$ of $E$ do not map nontrivially to $M[1,6]$. Applying Lemma \ref{lem:zero-action}, the strict admissible pair $(\mathbf{m},3)$ is not realized by a mixed brick. 

It remains to explain why no strict admissible pair is realized in this example. Suppose that $M'=(N',W',\eta')$ is a mixed brick in this example. By Corollary~\ref{cor::mixed-brick-summand}, every indecomposable direct summand of $N'$ is $E$-visible. 
By Lemma~\ref{lem:zero-action}, the direct sum $N'$ must have the following property: no nonzero endomorphism in $\End_B(N')$ acts as zero on $\Hom_B(E,N')$ by post-composition. 
We now ignore multiplicities for the moment and list the maximal collections of $E$-visible interval modules whose direct sums have this property. This is harmless because the following construction only depends on the support of $N'$. Once the possible supports are restricted, arbitrary multiplicities are encoded by nonnegative real coordinates in the corresponding cone.

The $E$-visible indecomposable modules are 
\[ 
M[1,1],\quad M[1,2],\quad M[2,2], \quad\text{and}\quad M[p,q]\text{ with }p=1,2,3,\ 3\leq q\leq6. 
\] 
Set $\mathcal{R}_p:=\{M[p,q]\mid 3\leq q\leq6\}\; (p=1,2,3)$. One finds that the module $M[1,1]$ cannot occur together with $M[1,2]$ or with any module in $\mathcal{R}_1$; the module $M[1,2]$ cannot occur together with any module in $\mathcal{R}_1\cup\mathcal{R}_2$; the module $M[2,2]$ cannot occur together with any module in $\mathcal{R}_2$. In each of these cases, the interval Hom formula gives a nonzero morphism between two summands of $N'$ which acts as zero on $\Hom_B(E,N')$. Consequently, the support of every $E$-faithful module must avoid the forbidden pairs listed above. It follows that its support is contained in one of the following five maximal collections.
\[
\mathcal{R}_1\cup\mathcal{R}_2\cup\mathcal{R}_3, \quad \{M[1,1]\}\cup\mathcal{R}_2\cup\mathcal{R}_3,\quad \{M[1,1],M[2,2]\}\cup\mathcal{R}_3,
\]
\[
\{M[2,2]\}\cup\mathcal{R}_1\cup\mathcal{R}_3,\quad \{M[1,2],M[2,2]\}\cup\mathcal{R}_3.
\]

We now show that none of these five maximal collections can support a strict admissible pair. 
Recall that $Q(\mathbf{x})=4r(\mathbf{x})-h(\mathbf{x})^2$ for any $\mathbf{x}\in \mathbb{R}_{\geq 0}^t$. 
For each such collection $\mathcal{R}$, we restrict $Q$ to the coordinates corresponding to the indecomposable modules in $\mathcal{R}$; equivalently, we set all coordinates outside $\mathcal{R}$ equal to zero. 
We denote the resulting quadratic form by $Q_{\mathcal{R}}$. 
By Theorem~\ref{thm:strict-pairs}, it is enough to show that $Q_{\mathcal{R}}(\mathbf{x})>0$, for every nonzero $\mathbf{x}\in\mathbb{R}_{\geq 0}^{\mathcal{R}}$. Here, the vector $\mathbf{x}$ is only a real weight vector on the chosen collection, not necessarily a multiplicity vector. 
This causes no problem, since proving positivity on the whole real cone is stronger than proving it only on integer multiplicity vectors.

\begin{itemize}
\item Set $\mathcal{R}=\mathcal{R}_1\cup\mathcal{R}_2\cup\mathcal{R}_3$. 
It forms a $3\times4$ rectangle in the interval coordinates, namely it consists of the modules $M[p,q]$ with $p=1,2,3$ and $q=3,4,5,6$. 
Let $x_{p,q}$ be the coordinate corresponding to $M[p,q]$, and put $x_p=x_{p,3}+x_{p,4}+x_{p,5}+x_{p,6}$, $x_q=x_{1,q}+x_{2,q}+x_{3,q}$. Although $r(\mathbf{x})$ was defined using the symmetrized matrix $\mathsf{G}$, we may compute it using the directed Hom matrix $\mathsf{C}$, since $\mathbf{x}^{\mathsf{T}}\mathsf{G}\mathbf{x}=\mathbf{x}^{\mathsf{T}}\mathsf{C}\mathbf{x}$. By the interval Hom formula, for modules in this rectangle one has
\[
\Hom_B(M[p,q],M[p',q'])\neq0
\quad\Longleftrightarrow\quad
p'\leq p,\ q'\leq q.
\]
Here, the middle condition $p\leq q'$ is automatic because $p\leq3\leq q'$. Then, $r(\mathbf{x})$ contains all same-row and same-column contributions. After subtracting the diagonal terms once, we have 
\[
r(\mathbf{x}) =
\sum_{p=1}^3\sum_{q=3}^6
\sum_{p'=1}^{p}\sum_{q'=3}^{q}
x_{p,q}x_{p',q'}
\geq
\frac12\sum_{p=1}^3 x_p^2
+
\frac12\sum_{q=3}^6 x_q^2.
\]
Since each module $M[p,q]$ in this rectangle contributes exactly one dimension to $\Hom_B(E,M[p,q])$, we have $h(\mathbf{x})=\sum_p x_p=\sum_q x_q$. Using Cauchy's inequality, we get
\[
r(\mathbf{x})
\geq
\frac12\left(\frac13+\frac14\right)h(\mathbf{x})^2
=
\frac7{24}h(\mathbf{x})^2
>
\frac14h(\mathbf{x})^2.
\]
It follows that $Q_{\mathcal{R}}(\mathbf{x})>0$ for every nonzero $\mathbf{x}\in\mathbb{R}_{\geq 0}^{\mathcal{R}}$.

\item Set $\mathcal{R}=\{M[1,1]\}\cup\mathcal{R}_2\cup\mathcal{R}_3$. By the interval Hom formula, there are no nonzero Hom-spaces in either direction between $M[1,1]$ and any module in $\mathcal{R}_2\cup\mathcal{R}_3$. Hence, $r(\mathbf{x})$ splits as the sum of the two parts. Let $h_1(\mathbf{x})$ be the contribution to $h(\mathbf{x})$ from $M[1,1]$, and let $h_2(\mathbf{x})$ be the contribution to $h(\mathbf{x})$ from $\mathcal{R}_2\cup\mathcal{R}_3$. The isolated module $M[1,1]$ gives $r_1(\mathbf{x})=h_1(\mathbf{x})^2$. The rectangle $\mathcal{R}_2\cup\mathcal{R}_3$ is a $2\times4$ rectangle, so the same strategy as in the previous case gives
\[
r_2(\mathbf{x})\geq
\frac12\left(\frac12+\frac14\right)h_2(\mathbf{x})^2
=
\frac38h_2(\mathbf{x})^2.
\]
Therefore, for any nonzero $\mathbf{x}\in\mathbb{R}_{\geq 0}^{\mathcal{R}}$, $h(\mathbf{x})=h_1(\mathbf{x})+h_2(\mathbf{x})>0$ gives 
\[
\begin{aligned}
Q_{\mathcal{R}}(\mathbf{x})
&\geq
4h_1(\mathbf{x})^2+\frac32h_2(\mathbf{x})^2-(h_1(\mathbf{x})+h_2(\mathbf{x}))^2\\
&=
\frac12\left((h_2(\mathbf{x})-2h_1(\mathbf{x}))^2+2h_1(\mathbf{x})^2\right)>0.
\end{aligned}
\]

\item Similar arguments as above apply to $\{M[2,2]\}\cup\mathcal{R}_1\cup\mathcal{R}_3$ and $\{M[1,2],M[2,2]\}\cup\mathcal{R}_3$; we leave the details to readers. It remains to consider $\mathcal{R}=\{M[1,1],M[2,2]\}\cup\mathcal{R}_3$. Let $x$ and $y$ be the coordinates corresponding to $M[1,1]$ and $M[2,2]$, respectively. 
Let $z_1,z_2,z_3,z_4$ be the coordinates corresponding to the four modules in $\mathcal{R}_3$. Put $u=x+y$ and $v=z_1+z_2+z_3+z_4$. Since the two isolated modules $M[1,1]$ and $M[2,2]$ have no nonzero Hom-spaces between them, their contribution to $r(\mathbf{x})$ is
\[
x^2+y^2\geq\frac12(x+y)^2=\frac12u^2.
\]
On the other hand, $\mathcal{R}_3$ is a $1\times4$ rectangle, so the rectangle estimate gives
\[
r_{\mathcal{R}_3}(\mathbf{x})
\geq
\frac12\left(1+\frac14\right)v^2
=
\frac58v^2.
\]
Since each module in this collection contributes exactly one dimension to the Hom-space from $E$, we have $h(\mathbf{x})=u+v>0$. Hence, for every nonzero $\mathbf{x}\in\mathbb{R}_{\geq 0}^{\mathcal{R}}$,
\[
\begin{aligned}
Q_{\mathcal{R}}(\mathbf{x})
&\geq
4\left(\frac12u^2+\frac58v^2\right)-(u+v)^2=(u-v)^2+\frac12v^2
>0.
\end{aligned}
\]
\end{itemize}

We conclude that, in this example, we have found a strict admissible pair, but no strict admissible pair is realized by a mixed brick. Moreover, we claim that the algebra $\Lambda=B[E]$ is brick-finite. If $M'=(N',W',\eta')$ is a mixed brick in $\mod \Lambda$, then $N'$ is $E$-faithful, and hence its support is contained in one of the five collections considered above. By Theorem~\ref{thm:strict-pairs}, only finitely many admissible pairs can be realized by mixed bricks, and all of them are critical. By Theorem~\ref{thm::critical-pair-one-orbit}, there are only finitely many isomorphism classes of mixed bricks in $\mod \Lambda$. Since $B=\Bbbk Q_6$ is representation-finite, it is brick-finite, and consequently, $\Lambda$ is brick-finite.

\subsection{Necessary and sufficient conditions for brick-finiteness}
We now summarise the preceding results as a necessary and sufficient criterion for the finiteness of mixed bricks. 

\begin{theorem}\label{thm:mixed-finite}
Suppose that $B$ is $E$-visible finite. The following conditions are equivalent.
\begin{enumerate}
\item There are only finitely many isomorphism classes of mixed bricks in $\mod \Lambda$.

\item There are only finitely many realized admissible pairs, and each of them is critical.

\item There are only finitely many pairs $(\mathbf{m},\ell)$ with $\mathbf{m}\neq0$, $1\leq\ell\leq h(\mathbf{m})$, and $\mathcal{L}(\mathbf{m},\ell)\neq\varnothing$, and every such pair is critical.
\end{enumerate}
If one of these conditions holds, then $\Lambda$ is brick-finite if and only if $B$ is brick-finite. 
\end{theorem}
\begin{proof}
By Theorem~\ref{thm:dim-conditions}, every mixed brick gives an admissible pair $(\mathbf{m},\ell)$. By Proposition~\ref{prop::mixed-bricks-orbits}, mixed bricks are classified, for each fixed pair $(\mathbf{m},\ell)$, by the $\Aut_B(N_{\mathbf{m}})$-orbits in $\mathcal{L}(\mathbf{m},\ell)$. Then, $(1)\Rightarrow (2)$ follows from Proposition~\ref{prop:strict-infinite-bricks}, and $(2)\Rightarrow (1)$ follows from Theorem~\ref{thm::critical-pair-one-orbit}. Next, $(2)$ and $(3)$ are obviously equivalent by our constructions. 
We note that every brick $N\in\mod B$ gives a brick $(N,0,0)$ in $\mod \Lambda$. Then, through a simple classification of bricks in $\mod \Lambda$, one concludes that $\Lambda$ is brick-finite if and only if $B$ is brick-finite.
\end{proof}

In the following, we give some restrictions on the extension module $E$ from the assumption that $\Lambda=B[E]$ is brick-finite. 

\begin{proposition}\label{prop:quotient-bricks}
Let $K$ be a proper $B$-submodule of $E$ and $q_K\colon E\twoheadrightarrow E/K$ be the canonical epimorphism. We define $\eta_K\colon\Bbbk\longrightarrow\Hom_B(E,E/K)$ by $\eta_K(1)=q_K$. Then,
\[
M_K:=(E/K,\Bbbk,\eta_K)
\]
is a mixed brick in $\mod \Lambda$. Moreover, $M_K\cong M_{K'}$ if and only if $K=K'$.
\end{proposition}
\begin{proof}
Since $q_K\neq0$, the map $\eta_K$ is injective and $\Im\eta_K=\Bbbk q_K$. If
$f\in\End_B(E/K)$ stabilizes $\Im\eta_K$, then $fq_K=\lambda q_K$ for some $\lambda\in\Bbbk$. The surjectivity of $q_K$ gives
$f=\lambda\operatorname{id}_{E/K}$. Then, Proposition~\ref{prop::mixed-brick-criterion} shows that $M_K$ is a
mixed brick. If $(f,g)\colon M_K\to M_{K'}$ is an isomorphism, the compatibility condition gives $fq_K=\lambda' q_{K'}$ for some $\lambda'\in\Bbbk^{\times}$. Taking kernels on both sides yields $K=K'$.
\end{proof}

Let $\operatorname{brick}\Lambda$ denote the set of bricks in $\mod\Lambda$, and denote its cardinality by $|\operatorname{brick}\Lambda|$.

\begin{corollary}\label{cor:submodule-bound}
If $\Lambda$ is brick-finite, then $E$ has only finitely many $B$-submodules. More precisely,
\[
|\operatorname{brick}\Lambda|
\geq
|\operatorname{brick}B|
+
|\operatorname{Sub}_B(E)|,
\]
where $\operatorname{Sub}_B(E)$ denotes the set of all $B$-submodules of $E$.
\end{corollary}
\begin{proof}
By Proposition~\ref{prop:quotient-bricks}, distinct proper submodules $K$ of $E$ give pairwise nonisomorphic mixed bricks $M_K$. The old bricks $(X,0,0)$ in $\mod B$, the new simple brick $(0,\Bbbk,0)$, and the mixed bricks $M_K$ are pairwise distinct types. Since the proper submodules of $E$ contribute $|\operatorname{Sub}_B(E)|-1$ mixed bricks, the stated inequality follows.
\end{proof}

\begin{proposition}
Suppose that $B$ is $E$-visible finite. If there are two epimorphisms $\varphi,\varphi'\colon E\twoheadrightarrow N$ with different kernels, then $\Lambda$ is brick-infinite.
\end{proposition}
\begin{proof}
Consider the linear map $\rho_{N}^\varphi \colon \End_B(N)\rightarrow \Hom_B(E,N)$ given by $f\mapsto f\circ \varphi$. Since $\varphi$ is an epimorphism, the map $\rho_{N}^\varphi $ is injective. Suppose that $\varphi'=f\circ \varphi$ for some $f\in\End_B(N)$. Since both $\varphi, \varphi'$ are surjective, we have $N=\varphi'(E)=f(\varphi(E))=f(N)$, and hence $f$ is an automorphism. It follows that $\ker\varphi'=\ker(f\circ \varphi)=\ker\varphi$, which contradicts the assumption. We obtain $\varphi'\notin\operatorname{Im}(\rho_{N}^\varphi )$, and then $\dim_{\Bbbk}\Hom_B(E,N)>
\dim_{\Bbbk}\End_B(N)$.

By Proposition~\ref{prop:quotient-bricks}, $(N,\Bbbk,\eta_\varphi)$ with $\eta_\varphi(1)=\varphi$ is a mixed brick in $\mod \Lambda$. We write $N\cong N_{\mathbf{m}}$ for some multiplicity vector $\mathbf{m}$, and then $(\mathbf{m}, 1)$ is a realized strict admissible pair.  Proposition~\ref{prop:strict-infinite-bricks} then implies that there are infinitely many isomorphism classes of mixed bricks.
\end{proof}

In order to prove brick-infiniteness of $\Lambda$, it is useful to have concrete criteria which guarantee the realization of a strict admissible pair. The next result gives such a criterion in terms of semibricks: a nonempty collection $\{X_i\mid i= 1,2,\ldots,s\}$ of pairwise nonisomorphic $B$-modules is called a \emph{semibrick} in $\mod B$ if each $X_i$ is a brick and $\Hom_B(X_i,X_j)=0$ for all distinct $1\le i,j\le s$. The point is that, when the $B$-module part $N$ is a direct sum of pairwise Hom-orthogonal $E$-visible bricks, its endomorphism algebra is diagonal. This makes it possible to choose a subspace of $\Hom_B(E,N)$ whose stabilizer is exactly the scalar endomorphisms.

\begin{proposition}\label{prop:semibrick-strict}
Let $I\subseteq\{1,2,\ldots,t\}$ be a nonempty subset such that $\{X_i\mid i\in I\}$ is a semibrick in $\mod B$. Set $\mathbf{m}_I:=\sum_{i\in I}\mathbf{e}_i$, where $\mathbf{e}_i$ denotes the $i$-th standard basis vector of $\mathbb{Z}^t$, and set $s:=|I|$ and $h(\mathbf{m}_I)=\sum_{i\in I}d_i$.
If there exists an integer $\ell$ such that $1\leq\ell\leq h(\mathbf{m}_I)-1$ and $s-1<\ell(h(\mathbf{m}_I)-\ell)$, then $(\mathbf{m}_I,\ell)$ is a strict admissible pair realized by a mixed brick. 
\end{proposition}
\begin{proof}
Set $N_{\mathbf{m}_I}:=\bigoplus_{i\in I}X_i$. We have 
\[
\End_B(N_{\mathbf{m}_I})\cong\prod_{i\in I}\End_B(X_i)\cong\Bbbk^s.
\]
In particular, $r(\mathbf{m}_I)=s$. Thus, the inequality $s-1<\ell(h(\mathbf{m}_I)-\ell)$ says precisely that $(\mathbf{m}_I,\ell)$ is a strict admissible pair. It remains to show that $(\mathbf{m}_I,\ell)$ is realized by a mixed brick. An endomorphism of $N_{\mathbf{m}_I}$ is given by a tuple $(\lambda_i)_{i\in I}$, and its action on $\Hom_B(E,N_{\mathbf{m}_I})$ by post-composition is scalar multiplication by $\lambda_i$ on the direct summand $\Hom_B(E,X_i)$.

We now show that the subspaces stabilised by some non-scalar endomorphism of $N_{\mathbf{m}_I}$ do not cover the whole Grassmannian. Let $f\in\End_B(N_{\mathbf{m}_I})$ be a non-scalar endomorphism. Under the identification $\End_B(N_{\mathbf{m}_I})\cong\Bbbk^s$, the endomorphism $f$ acts on each direct summand $\Hom_B(E,X_i)$ of $\Hom_B(E,N_{\mathbf{m}_I})$ by a scalar.
Since $f$ is non-scalar, not all these scalars are equal.
Grouping the summands $\Hom_B(E,X_i)$ according to these scalars, we obtain a nontrivial decomposition
\[
\Hom_B(E,N_{\mathbf{m}_I})=K_1\oplus K_2\oplus\cdots\oplus K_m
\]
with $m\geq2$, such that $f$ acts on each $K_j$ by a scalar and these scalars are pairwise distinct.
If an $\ell$-dimensional subspace $L\subseteq\Hom_B(E,N_{\mathbf{m}_I})$ satisfies $f_*(L)\subseteq L$, then $L$ is stable under this semisimple operator $f$, and hence $L=(L\cap K_1)\oplus\cdots\oplus(L\cap K_m)$. We set $\ell_j:=\dim(L\cap K_j)$, then $\ell_1+\ell_2+\cdots+\ell_m=\ell$, and $L$ is obtained by choosing subspaces $L_j\in\Gr_{\ell_j}(K_j)$ and taking $L=L_1\oplus L_2\oplus\cdots\oplus L_m$.
For this fixed non-scalar endomorphism $f$, the set of $\ell$-dimensional subspaces stable under $f$ is contained in the finite union, over all tuples $(\ell_1,\ell_2,\ldots,\ell_m)$ with $\ell_1+\ell_2+\cdots+\ell_m=\ell$, of the images of the natural morphisms
\[
\begin{aligned}
\Gr_{\ell_1}(K_1)\times\cdots\times\Gr_{\ell_m}(K_m)
&\longrightarrow
\Gr_\ell(\Hom_B(E,N_{\mathbf{m}_I}))\\
(L_1,L_2,\ldots,L_m)&\longmapsto L_1\oplus L_2\oplus\cdots\oplus L_m .
\end{aligned}
\]
This finite union is closed, because each source is projective and the image of a projective variety under a morphism is closed; see \cite[Chapter~I, Section 5, Theorem~1.10]{Shafarevich-2013}. As $f$ varies over all non-scalar endomorphisms, these decompositions are indexed by the finitely many nontrivial partitions of the finite set $I$, because they are obtained by grouping the finitely many summands $\Hom_B(E,X_i)$.

It remains to show that this closed subset is proper.
Let $k_j:=\dim K_j$. For a fixed tuple $(\ell_1,\ell_2,\ldots,\ell_m)$, the dimension of the source is $\sum_{j=1}^m\ell_j(k_j-\ell_j)$. Since $\ell=\sum_{j=1}^m\ell_j$ and $h(\mathbf{m}_I)=\sum_{j=1}^m k_j$, we have
\[
\ell(h(\mathbf{m}_I)-\ell)-\sum_{j=1}^m\ell_j(k_j-\ell_j)
=
\sum_{p\neq q}\ell_p(k_q-\ell_q).
\]
The right-hand side is positive because the decomposition is nontrivial and $1\leq\ell\leq h(\mathbf{m}_I)-1$. Thus, every such image has dimension strictly smaller than $\dim\Gr_\ell(\Hom_B(E,N_{\mathbf{m}_I}))=\ell(h(\mathbf{m}_I)-\ell)$.
Consequently, the set of $\ell$-dimensional subspaces stable under a non-scalar endomorphism is contained in a proper closed subset of $\Gr_\ell(\Hom_B(E,N_{\mathbf{m}_I}))$. There are only finitely many decompositions of $\Hom_B(E,N_{\mathbf{m}_I})$ obtained by grouping the summands $\Hom_B(E,X_i)$.
Therefore, the subspaces stable under some non-scalar endomorphism of $N_{\mathbf{m}_I}$ are contained in a finite union of proper closed subsets of $\Gr_\ell(\Hom_B(E,N_{\mathbf{m}_I}))$.
Since $\Gr_\ell(\Hom_B(E,N_{\mathbf{m}_I}))$ is irreducible, this finite union cannot be the whole Grassmannian.
Hence, there exists a subspace $L\in\Gr_\ell(\Hom_B(E,N_{\mathbf{m}_I}))$ such that $\Stab_{\End_B(N_{\mathbf{m}_I})}(L)=\Bbbk\operatorname{id}_{N_{\mathbf{m}_I}}$.

Choose such a subspace $L$, and let $\iota\colon L\hookrightarrow \Hom_B(E,N_{\mathbf{m}_I})$ be the natural inclusion.
By Proposition~\ref{prop::mixed-brick-criterion}, the triple $(N_{\mathbf{m}_I},L,\iota)$ is a mixed brick.
Thus, $(\mathbf{m}_I,\ell)$ is realized by a mixed brick.
\end{proof}

\begin{corollary}\label{cor:visible-obstructions}
Suppose that $\Lambda=B[E]$ is brick-finite. Then the following statements hold.
\begin{enumerate}
\item If $X_i$ is an $E$-visible brick, then $\dim_{\Bbbk}\Hom_B(E,X_i)=1$. 

\item Every $E$-visible semibrick contains at most three elements.
\end{enumerate}
\end{corollary}
\begin{proof}
If $X_i$ is an $E$-visible brick with $d_i\geq2$, then applying Proposition~\ref{prop:semibrick-strict} to $I=\{i\}$ and $\ell=1$, we obtain a realized strict admissible pair. Proposition~\ref{prop:strict-infinite-bricks} would then imply that $\Lambda$ is brick-infinite.

Let $\{X_i\mid i\in I\}$ be an $E$-visible semibrick and set $s:=|I|$. Suppose, for a contradiction, that $s\geq4$. We have $d_i\geq1$ and $h(\mathbf{m}_I)=\sum_{i\in I}d_i\geq s$. Set $\ell:=\left\lfloor h(\mathbf{m}_I)/2\right\rfloor$, we have $1\leq\ell\leq h(\mathbf{m}_I)-1$. Moreover,
\[
\begin{aligned}
\ell\bigl(h(\mathbf{m}_I)-\ell\bigr)=
\left\lfloor\frac{h(\mathbf{m}_I)^2}{4}\right\rfloor\geq
\left\lfloor\frac{s^2}{4}\right\rfloor
>
s-1,
\end{aligned}
\]
holds for every $s\geq4$. Then, Proposition~\ref{prop:semibrick-strict} and Proposition~\ref{prop:strict-infinite-bricks} show that $\Lambda$ is brick-infinite, contradicting our assumption.
\end{proof}

The preceding restrictions, together with the finiteness of the submodule lattice of $E$, strongly restrict the Krull--Schmidt decomposition of $E$.

\begin{proposition}\label{prop:E-decomposition}
Suppose that $\Lambda=B[E]$ is brick-finite, and write $E\cong E_1^{a_1}\oplus\cdots\oplus E_s^{a_s}$, where $E_1,\ldots,E_s$ are pairwise nonisomorphic indecomposable $B$-modules and $a_i\geq1$. Then, the following statements hold.
\begin{enumerate}
\item We have $a_i=1$ for every $1\leq i\leq s$, and $s\leq3$.

\item If $i\neq j$, then there is no nonzero homomorphism from a subfactor of $E_i$ to a subfactor of $E_j$. In particular, $\Hom_B(E_i,E_j)=\Hom_B(E_j,E_i)=0$ for any $i\neq j$.
\end{enumerate}
\end{proposition}
\begin{proof}
By Corollary~\ref{cor:submodule-bound}, the module $E$ has only finitely many $B$-submodules. Then, \cite[Theorem~2.2]{Iovanov-Koffi-2022} implies that $E$ has no subfactor isomorphic to $S^{\oplus 2}$ for any simple $B$-module $S$. Suppose that $a_i\geq2$ for some $i$. Choose a simple composition factor $S$ of $E_i$, and choose submodules $V\subseteq U\subseteq E_i$ such that $U/V\cong S$. Using two copies of $E_i$ in $E_i^{a_i}$, we obtain a subfactor $S^{\oplus 2}$ of $E$, which is a contradiction. Therefore, $a_i=1$ for every $1\leq i\leq s$.

Suppose that distinct summands $E_i$ and $E_j$ have a common simple composition factor $S$. Choose subfactors $U_i/V_i$ of $E_i$ and $U_j/V_j$ of $E_j$ that are both isomorphic to $S$. Then, $(U_i\oplus U_j)/(V_i\oplus V_j)\cong S^{\oplus 2}$ is a subfactor of $E$, again a contradiction. Hence, distinct summands $E_i$ and $E_j$ have no common simple composition factor. More generally, let $U/V$ and $U'/V'$ be subfactors of $E_i$ and $E_j$, respectively, where $i\neq j$. If there is a nonzero homomorphism $U/V\longrightarrow U'/V'$, then its nonzero image would have a simple composition factor occurring in both $E_i$ and $E_j$, a contradiction. Therefore, every such homomorphism is zero. 

Choose a simple quotient $S_i$ of each $E_i$. Since distinct $E_i$ have no common simple composition factor, the simple modules $S_1,\ldots,S_s$ are pairwise nonisomorphic. Moreover, the composite epimorphism $E\twoheadrightarrow E_i\twoheadrightarrow S_i$ shows that every $S_i$ is $E$-visible. Thus, $\{S_1,\ldots,S_s\}$ is an $E$-visible semibrick and  Corollary~\ref{cor:visible-obstructions}$(2)$ gives $s\leq3$.
\end{proof}

\begin{corollary}
Suppose that $\Lambda=B[E]$ is brick-finite. In the Gabriel quiver of $\Lambda$, the extension vertex is a source from which at most three arrows start, and no two of these arrows have the same target.
\end{corollary}
\begin{proof}
The number of arrows from the extension vertex to the vertex
corresponding to a simple $B$-module $S$ is the multiplicity of $S$
in $\top E$. Write $\top E\cong S_1^{a_1}\oplus\cdots\oplus S_q^{a_q}$, where the $S_i$ are pairwise nonisomorphic and $a_i\geq1$. By
Corollary~\ref{cor:submodule-bound}, the module $E$ has only finitely
many submodules. Hence, \cite[Theorem~2.2]{Iovanov-Koffi-2022}
implies that $E$ has no subfactor isomorphic to $S^{\oplus2}$ for
any simple $B$-module $S$. Therefore, $a_i=1$ for every $i$.
The set $\{S_1,\ldots,S_q\}$ is an $E$-visible semibrick, so
Corollary~\ref{cor:visible-obstructions}$(2)$ gives $q\leq3$.
\end{proof}

\subsection{The second brick--Brauer--Thrall conjecture}
Recall that the second brick--Brauer--Thrall conjecture (2bBTC) predicts that every brick-infinite algebra admits infinitely many bricks of the same dimension. For one-point extension algebras, a realized strict admissible pair produces such a family directly.

\begin{proposition}\label{prop:strict-second-bbt}
Let $\Lambda=B[E]$. If $(\mathbf{m},\ell)$ is a realized strict admissible pair, then $\Lambda$ admits infinitely many mixed bricks with the same dimension vector.
\end{proposition}
\begin{proof}
By Proposition~\ref{prop:strict-infinite-bricks}, there are infinitely many pairwise nonisomorphic mixed bricks realizing $(\mathbf{m},\ell)$. Every such mixed brick has $B$-module part isomorphic to $N_{\mathbf{m}}$ and vector-space part of dimension $\ell$. Hence, they all have dimension vector $\bigl(\underline{\dim}_B N_{\mathbf{m}},\ell\bigr)$.
\end{proof}

Every brick in $\mod B$ gives a brick of the form $(N,0,0)$ in $\mod \Lambda$. Thus, if $B$ is brick-infinite and already satisfies the 2bBTC, then so does $\Lambda$. The essential remaining case is therefore when $B$ is brick-finite. In this case, the old bricks in $\mod B$ contribute only finitely many isomorphism classes, while $(0,\Bbbk,0)$ is the only brick supported at the extension vertex. It is reasonable to give the following conjecture, which is equivalent to the 2bBTC.

\begin{conjecture}\label{conj:critical-to-strict}
Suppose that $B$ is $E$-visible finite and brick-finite. If $\Lambda$ is brick-infinite, then there exists a realized strict admissible pair. Equivalently, if infinitely many critical admissible pairs are realized, then at least one strict admissible pair is realized.
\end{conjecture}

We now verify Conjecture~\ref{conj:critical-to-strict} in an elementary family of algebras.

\begin{theorem}\label{thm:interval-classification}
Let $B=\Bbbk Q_n$ as in Subsection \ref{subsec-example}, and take $E=M[a,b]$ for $1\leq a\leq b\leq n$. Then, $\Lambda=B[E]$ is brick-finite if and only if
\begin{equation}\label{Dynkin-class}
\frac1a+\frac1{b-a+1}>\frac12.
\end{equation}
\end{theorem}
\begin{proof}
Let $e_{\leq b}:=e_0+e_1+\cdots+e_b$ and set $\Pi_{a,b}:=\Lambda/\Lambda(1-e_{\leq b})\Lambda$. Then, $\Pi_{a,b}$ is the path algebra of the acyclic quiver obtained from $1\to2\to\cdots\to b$ by adjoining the arrow $0\to a$. Its underlying graph is the three-armed tree $T_{2,a,b-a+1}$. By the well-known classification of three-armed trees, $T_{2,a,b-a+1}$ is Dynkin if and only if \eqref{Dynkin-class} holds. 

Suppose first that \eqref{Dynkin-class} does not hold. Then, $\Pi_{a,b}$ is a hereditary algebra of non-Dynkin type and hence brick-infinite, see, for example, \cite[Remark 2.9]{Mousavand-biserial}. Since bricks over a quotient algebra give rise to bricks over $\Lambda$, it follows that $\Lambda$ is brick-infinite.

Suppose that \eqref{Dynkin-class} holds. Then, $\Pi_{a,b}$ is a representation-finite hereditary algebra. Let $M=(N,W,\eta)$ be a mixed brick over $\Lambda$. By Corollary~\ref{cor::mixed-brick-summand} and  the interval Hom formula, $M$ is actually a $\Pi_{a,b}$-module. Therefore, $\Lambda$ has only finitely many mixed bricks, and moreover, $\Lambda$ is brick-finite by Theorem \ref{thm:mixed-finite}.
\end{proof}

Since every brick-infinite one-interval extension admits an affine hereditary quotient algebra, we may carry infinitely many bricks of one fixed dimension vector. Then, the preceding classification allows us to verify Conjecture~\ref{conj:critical-to-strict} for this case.

\begin{corollary}
Let $B=\Bbbk Q_n$ and $E=M[a,b]$ for $1\leq a\leq b\leq n$. Then, Conjecture~\ref{conj:critical-to-strict} holds for $\Lambda=B[E]$.
\end{corollary}
\begin{proof}
Suppose that $\Lambda$ is brick-infinite. Set $c:=b-a+1$. By Theorem~\ref{thm:interval-classification}, \eqref{Dynkin-class} does not hold and then $a,c\geq3$. Moreover, if $a=3$, then $c\geq6$; if $c=3$, then $a\geq6$; otherwise $a,c\geq4$. We may choose $p\leq a$ and $q\leq c$ such
that
\[
(p,q)\in\{(3,6),(4,4),(6,3)\}.
\]
Let $\Gamma$ be the quotient of $\Lambda$ obtained by deleting all vertices except the new vertex $0$ and the old vertices $a-p+1,a-p+2,\ldots,a+q-1$. Then, $\Gamma$ is the path algebra of an acyclic quiver whose graph is $T_{2,p,q}$, which is a hereditary algebra of affine type $\widetilde{\mathbb{E}}_7$ or $\widetilde{\mathbb{E}}_8$. The structure theory of regular $\Gamma$-modules implies that $\Gamma$ has infinitely many homogeneous tubes in the Auslander-Reiten quiver of $\Gamma$. Choosing one simple regular module from each homogeneous tube gives infinitely many pairwise non-isomorphic bricks with the same dimension vector; see, for example, \cite[XIII.2]{SS-II} or \cite[Chapter~3]{RingelTame}. These bricks must have nonzero components both at the extension vertex and at the old vertices: otherwise they would form an infinite family of bricks over a Dynkin quiver of type $\mathbb{A}$, or would all be supported at the single vertex $0$. Hence, all of them are mixed. Their inflations along $\Lambda\twoheadrightarrow\Gamma$ give infinitely many pairwise nonisomorphic mixed $\Lambda$-bricks with the same dimension vector.

Since $B$ is representation-finite, only finitely many isomorphism classes can occur among the $B$-module parts of these mixed bricks. After passing to an infinite subfamily, all their $B$-module parts are isomorphic to some fixed module $N_{\mathbf{m}}$. Their dimensions at the extension vertex also have a common value $\ell$. Thus, $(\mathbf{m},\ell)$ is realized by infinitely many pairwise nonisomorphic mixed bricks. It cannot be critical by Theorem~\ref{thm::critical-pair-one-orbit}. Hence, $(\mathbf{m},\ell)$ is a realized strict admissible pair, proving
Conjecture~\ref{conj:critical-to-strict}.
\end{proof}

This provides a first verification of Conjecture~\ref{conj:critical-to-strict} for one-point extensions of path algebras of type $\mathbb{A}$. We expect the conjecture to hold more generally for arbitrary extension modules over path algebras of type $\mathbb{A}$; a systematic treatment of that problem will be developed separately.

\section{Applications}\label{Section-5}
\subsection{Projective extension modules}
We now consider the special case where the extension module is projective. Let $P$ be a projective right $B$-module and $\Lambda=B[P]$. 

One-point extensions by projective modules have previously been studied
from the viewpoint of $\tau$-tilting theory. In \cite{Suarez-projective-one-point}, the author proved that the extension and restriction functors transfer support $\tau$-tilting modules between $B$ and $\Lambda$, and that the extension functor induces a full embedding between the corresponding support $\tau$-tilting quivers. Our purpose here is not to extend these functorial results. Instead, we explain how the mixed brick framework developed in this paper goes beyond the maximal framing construction.

With our convention for right modules, we define the restriction functor $\mathcal{R}\colon\mod \Lambda\rightarrow\mod B$ by $\mathcal{R}(N,W,\eta):=N$ and the extension functor $\mathcal{E}\colon\mod B\rightarrow\mod \Lambda$ by 
\[
\mathcal{E}(N) := \bigl(
N,\Hom_B(P,N),\operatorname{id}_{\Hom_B(P,N)} \bigr). 
\]
On morphisms, we set $\mathcal{R}(f,g):=f$ and $\mathcal{E}(f):=(f,f_*)$. Thus, $\mathcal{E}(N)$ is obtained by equipping $N$ with the maximal framing space $\Hom_B(P,N)$. The following observation explains the behavior of mixed bricks under this construction.

\begin{proposition}\label{prop:max-framing}
Let $P$ be a projective right $B$-module and $\Lambda=B[P]$.
Then, the following statements hold.
\begin{enumerate}
\item The functor $\mathcal{E}$ is an exact and fully faithful right
adjoint to $\mathcal{R}$.

\item For all $N,N'\in\mod B$, there is a natural isomorphism $\Hom_\Lambda\bigl(\mathcal{E}(N),\mathcal{E}(N')\bigr) \cong \Hom_B(N,N')$. In particular, $\mathcal{E}(N)$ is a brick if and only if $N$ is a brick.
\end{enumerate}
\end{proposition}
\begin{proof}
Let $M=(N_0,W,\eta)\in\mod \Lambda$ and $N\in\mod B$. A morphism
\[
(f,g)\colon
(N_0,W,\eta)
\longrightarrow
\bigl(N,\Hom_B(P,N),\operatorname{id}_{\Hom_B(P,N)}\bigr)
\]
satisfies $f_*\circ\eta=g$. Hence, $g$ is uniquely determined by the $B$-module homomorphism $f\colon N_0\to N$. This gives a natural isomorphism $\Hom_\Lambda\bigl(M,\mathcal{E}(N)\bigr) \cong \Hom_B\bigl(\mathcal{R}(M),N\bigr)$, and therefore $(\mathcal{R},\mathcal{E})$ is an adjoint pair. Taking $M=\mathcal{E}(N_0)$ shows that $\mathcal{E}$ is fully faithful. Moreover, $\mathcal{E}$ is exact because $P$ is projective and hence $\Hom_B(P,-)$ is exact. The remaining assertions follow immediately.
\end{proof}

Proposition~\ref{prop:max-framing} shows that maximal framing preserves the entire endomorphism algebra of $N$ and therefore cannot turn a non-brick $B$-module into a mixed brick in $\mod \Lambda$. However, our mixed brick framework allows arbitrary nonzero partial framings $0\neq L\subseteq\Hom_B(P,N)$. A suitable partial framing may eliminate the non-scalar endomorphisms of $N$, even when maximal framing does not produce a brick. We give an example as follows.

\begin{example}
Let $B=\Bbbk$, $P=\Bbbk^3$ and $\Lambda=B[P]$. Then, $\Lambda$ is the path algebra of the 3-arrow Kronecker quiver, with the extension vertex as a source. Let $N=\Bbbk^2$. Since $\End_B(N)\cong\operatorname{Mat}_{2\times2}(\Bbbk)$, the module $N$ is not a brick, and hence its maximal framing $\mathcal{E}(N)$ is not a mixed brick by Proposition~\ref{prop:max-framing}. 

Choose bases $\{p_1,p_2,p_3\}$ of $P$ and $\{v_1,v_2\}$ of $N$. Define
$\epsilon_1,\epsilon_2\in\Hom_{\Bbbk}(P,N)$ by $\epsilon_1(p_1)=\epsilon_1(p_2)=0$, $\epsilon_1(p_3)=v_2$ and $\epsilon_2(p_1)=v_2$, $\epsilon_2(p_2)=v_1$, $\epsilon_2(p_3)=0$. Set $L:=\Bbbk\epsilon_1+\Bbbk\epsilon_2$. Let $f\in\End_B(N)$ be given by
$f(v_1)=av_1+cv_2$ and $f(v_2)=bv_1+dv_2$, and suppose that
$f_*(L)\subseteq L$. Since $f_*(\epsilon_1)(p_1)=f_*(\epsilon_1)(p_2)=0$ and $f_*(\epsilon_1)(p_3)=bv_1+dv_2$, the condition $f_*(\epsilon_1)\in L$ implies $b=0$. Moreover, $f_*(\epsilon_2)\in L$ implies, by evaluating at $p_1$ and $p_2$, that $d=a$ and $c=0$. Thus, $f=a\operatorname{id}_{N}$, proving that $\Stab_{\End_B(N)}(L)=\Bbbk\operatorname{id}_{N}$. Hence, Proposition~\ref{prop::mixed-brick-criterion} shows that $(N,L,\iota_L)$ is a mixed brick.

We note that the unique indecomposable $B$-module is $X_1=\Bbbk$. Then,  $\mathbf{m}=(2)$, $h(\mathbf{m})=6$, $r(\mathbf{m})=4$, $\ell=2$. It turns out that $(\mathbf{m},2)$ is a realized strict admissible pair. Consequently, $\Lambda$ admits infinitely many pairwise nonisomorphic mixed bricks with dimension vector $(2,2)$. By contrast, the maximal framing corresponds to $\ell=h(\mathbf{m})=6$, which gives a non-admissible pair.
\end{example}

\subsection{Simple extension modules}
We consider one-point extensions by simple modules. In \cite{Gao-simple-one-point}, the author studied $\tau$-tilting modules over one-point extensions by a simple module corresponding to a source vertex, obtaining formulas for the numbers of (support) $\tau$-tilting modules. Here, we translate the simple-extension setting into our mixed brick framework.

Let $S$ be a simple right $B$-module and $\Lambda=B[S]$. For every $N\in\mod B$, the vector space $\Hom_B(S,N)$ records the multiplicity space of $S$ in $\soc N$. A mixed brick over $\Lambda$ may be viewed as an $S$-socle-framed module $(N,L,\iota_L)$ with $0\neq L\subseteq\Hom_B(S,N)$, where $\iota_L$ is the natural inclusion. The situation is particularly restrictive when $S$ is injective, i.e., $S$ corresponds to a source vertex in the quiver of $B$.

\begin{proposition}\label{prop:injective-simple}
Let $S$ be a simple injective right $B$-module and let $\Lambda=B[S]$. Then, up to isomorphism, the only mixed brick in $\mod \Lambda$ is $P_0:=(S,\Bbbk,\eta_0),\;\eta_0(1)=\operatorname{id}_S$, the indecomposable projective $\Lambda$-module associated with the extension
vertex. 
\end{proposition}
\begin{proof}
Since $S$ is an injective simple, $S$ is the unique $S$-visible indecomposable $B$-module. Let $M=(N,W,\eta)$ be a mixed brick. By Corollary~\ref{cor::mixed-brick-summand}, $N\cong S^{\oplus m}$ for some
$m\geq1$. Set $\ell:=\dim_{\Bbbk} W$. Then, $r(\mathbf{m})=m^2$, $h(\mathbf{m})=m$, and Theorem \ref{thm:dim-conditions} forces $m=1$ and $\ell=1$. Applying Proposition~\ref{prop:quotient-bricks} to the proper submodule $0\subsetneq S$, $P_0=(S,\Bbbk,\eta_0)$ is a mixed brick. Since $\dim_{\Bbbk}\Hom_B(S,S)=1$, every mixed brick with $B$-module part isomorphic to $S$ is isomorphic to $P_0$. 
\end{proof}

In the setting of Proposition~\ref{prop:injective-simple}, the only $B$-submodules of $S$ are $0$ and $S$, and the bricks in $\mod \Lambda$ consist precisely of the bricks inherited from $\mod B$, the simple module supported at the extension vertex, and the unique mixed brick $P_0$. Hence, whenever these sets are finite, $|\operatorname{brick}\Lambda| = |\operatorname{brick}B| + |\operatorname{Sub}_B(S)|$. This may be viewed as a brick counterpart to the explicit classification and counting results for $\tau$-tilting modules obtained in \cite{Gao-simple-one-point}.

The injective hypothesis in Proposition~\ref{prop:injective-simple} is essential: for an arbitrary simple extension module $S$, the $S$-visible region may contain several indecomposable modules. The following example gives a complete description in a basic hereditary case.

\begin{example}
Let $B=\Bbbk Q_n$ as in Subsection \ref{subsec-example} and take $S=M[a,a]$ with $1\leq a\leq n$. The $S$-visible indecomposable $B$-modules are precisely $M[c,a], 1\leq c\leq a$. Let $(\mathbf{m},\ell)$ be a realized admissible pair with $\mathbf{m}=(m_1,\ldots,m_a)\in \mathbb{Z}_{\geq 0}^a$. Then, $h(\mathbf{m})=m_1+\cdots +m_a$ and, if $h(\mathbf{m})\geq2$, then
\[
r(\mathbf{m})=\sum_{c\geq d}m_cm_d
\geq
\frac{h(\mathbf{m})^2+h(\mathbf{m})}{2}
>
\left\lfloor\frac{h(\mathbf{m})^2}{4}\right\rfloor+1,
\]
a contradiction. Hence,
$h(\mathbf{m})=1$, and this forces $r(\mathbf{m})=\ell=1$. One finds that $(\mathbf{m},1)$ is critical. Since the unique nonzero subspace of $\Hom_B(S,M[c,a])$ has scalar stabilizer, $(\mathbf{m},1)$ is realized by $\bigl(M[c,a],\Bbbk,\eta_c\bigr)$ with $0\neq \eta_c\colon \Bbbk\rightarrow\Hom_B(S,M[c,a])$, for each $1\leq c\leq a$. It follows from Theorem \ref{thm::critical-pair-one-orbit} that $\Lambda=B[S]$ has exactly $a$ isomorphism classes of mixed bricks. Using Theorem \ref{thm:interval-classification}, $\Lambda$ is brick-finite and $|\operatorname{brick}\Lambda| = |\operatorname{brick}B| + a+1$.
\end{example}

\subsection{Semibricks and epibricks}
In \cite{GaoXie}, the authors used semibricks to construct support $\tau$-tilting modules over one-point extensions $\Lambda=B[E]$. We explain how their construction appears in our framework and contrast it with the mixed brick part of $\mod \Lambda$.

Let $S_0:=(0,\Bbbk,0)$ be the simple module supported at the extension vertex. For $X\in\mod B$, write $\overline X:=(X,0,0)$, and $\overline{\mathcal{S}}:=\{\overline X\mid X\in\mathcal{S}\}$ for a collection $\mathcal{S}$ of $B$-modules.

\begin{proposition}\label{prop:boundary-semibricks}
Let $\mathcal{S}$ be a semibrick in $\mod B$. Then, both $\overline{\mathcal{S}}$ and $\overline{\mathcal{S}}\cup\{S_0\}$ are semibricks in $\mod \Lambda$. Moreover, no semibrick in $\mod \Lambda$ contains both $S_0$ and a mixed brick.
\end{proposition}
\begin{proof}
For $X,Y\in\mod B$, the compatibility condition for morphisms of triple notations gives $\Hom_\Lambda(\overline X,\overline Y) \cong \Hom_B(X,Y)$ and $\Hom_\Lambda(S_0,\overline X) = \Hom_\Lambda(\overline X,S_0)=0$. Hence, $\overline{\mathcal{S}}$ and $\overline{\mathcal{S}}\cup\{S_0\}$ are semibricks in $\mod \Lambda$.

Let $M=(N,W,\eta)$ be a mixed brick. Since $\eta$ is injective, a morphism $(f,g)\colon S_0\rightarrow M$ satisfies $\eta \circ g=0$. Hence $g=0$, and $\Hom_\Lambda(S_0,M)=0$. On the other hand, a morphism $(f,g)\colon M\rightarrow S_0$ is determined by an arbitrary linear map $g\colon W\rightarrow\Bbbk$, since its $B$-module component is necessarily zero and the compatibility condition is automatic. Thus, $\Hom_\Lambda(M,S_0) \cong \Hom_{\Bbbk}(W,\Bbbk)$. Since $M$ is mixed, one has $W\neq0$, and hence $\Hom_\Lambda(M,S_0)\neq0$. Therefore, $S_0$ and $M$ cannot occur in the same semibrick.
\end{proof}

The first assertion of Proposition~\ref{prop:boundary-semibricks} recovers, in the triple notation, the semibrick construction in \cite[Proposition~1.1]{GaoXie}. More precisely, they
proved that if $\mathcal{S}$ is a left-finite semibrick in $\mod B$, then $\overline{\mathcal{S}}\cup\{S_0\}$ is a left-finite semibrick in $\mod \Lambda$ (\cite[Proposition~3.4]{GaoXie}), and then produced support $\tau$-tilting $\Lambda$-modules associated with both $\overline{\mathcal{S}}$ and $\overline{\mathcal{S}}\cup\{S_0\}$. Thus, their construction uses only the boundary cases $W=0$ and $N=0$.

Let $M=(N,W,\eta)$ and $M'=(N',W',\eta')$ be mixed bricks. Since $\eta$ and $\eta'$ are injective, a $B$-module homomorphism
$f\colon N\to N'$ extends uniquely to a morphism $(f,g)\colon M\to M'$ if and only if $f_*(\operatorname{Im}\eta) \subseteq \operatorname{Im}\eta'$. Consequently,
\[
\Hom_\Lambda(M,M')
\cong
\left\{
f\in\Hom_B(N,N')
\;\middle|\;
f_*(\operatorname{Im}\eta)
\subseteq
\operatorname{Im}\eta'
\right\}.
\]
Hence, a collection of mixed bricks is a semibrick precisely when this space vanishes for every ordered pair of distinct members. This
additional Hom-orthogonality condition has no counterpart in the
boundary construction. 

\begin{example}
Let $B=\Bbbk Q_2$ as in Subsection \ref{subsec-example} and set $E=M[1,2]$, $S_i=M[i,i]$ for $i=1,2$. Then, $\Lambda=B[E]$ is the path algebra of $0\rightarrow1\rightarrow2$. One finds that $\mathcal{S}=\{S_2\}$ is a semibrick in $\mod B$, and then $\{\overline{S_2}\}$, $\{\overline{S_2},S_0\}$ are the boundary
semibricks in $\mod \Lambda$, mentioned in Proposition~\ref{prop:boundary-semibricks}.

By Proposition~\ref{prop:quotient-bricks}, the module $M_{S_2}:=(S_1,\Bbbk,\eta_{S_2})$ is a mixed brick. Moreover, $\Hom_\Lambda(\overline{S_2},M_{S_2}) \cong \Hom_B(S_2,S_1) = 0$, whereas
\[
\Hom_\Lambda(M_{S_2},\overline{S_2})
\cong
\{f\in\Hom_B(S_1,S_2)\mid f \circ q_{S_2}=0\}
=
0.
\]
It follows that $\{\overline{S_2},M_{S_2}\}$ is a semibrick in $\mod \Lambda$. 
\end{example}

Recall that a collection of bricks is called an \emph{epibrick} if every morphism between two bricks is either zero or an epimorphism. In \cite[Theorem~3.5]{LiGaoEpibricks}, the authors proved that if $\mathcal{E}$ is an epibrick in $\mod B$, then $\overline{\mathcal{E}} :=\{\overline X\mid X\in\mathcal{E}\}$ and $\overline{\mathcal{E}}\cup\{S_0\}$ are epibricks in $\mod \Lambda$. Similar to the semibrick construction discussed above, these are boundary epibricks consisting only of old bricks and, possibly, the simple module $S_0$ at the extension vertex.

Our canonical quotient construction in Proposition~\ref{prop:quotient-bricks} produces a different class of epibricks lying entirely in the mixed setting. Recall that $M_K=(E/K,\Bbbk,\eta_K)$ and $q_K\colon E\twoheadrightarrow E/K$ is the canonical epimorphism. Let $K$ and $K'$ be proper $B$-submodules of $E$. A morphism $(f,\lambda):M_K\rightarrow M_{K'}$ satisfies $f \circ q_K=\lambda q_{K'}$. If $\lambda=0$, then $f=0$ because $q_K$ is surjective. If
$\lambda\neq0$, $(f,\lambda)$ exists if and only if $K\subseteq K'$; this inclusion is precisely the condition that $q_{K'}$
factor through $q_K$. Writing $q_{K,K'}:E/K\rightarrow E/K'$ for the induced canonical epimorphism, every morphism $(f,\lambda)$ is of the form $(\lambda q_{K,K'},\lambda)$. Consequently,
\[
\Hom_\Lambda(M_K,M_{K'})
\cong
\begin{cases}
\Bbbk, & K\subseteq K',\\
0, & K\not\subseteq K'.
\end{cases}
\]
Every nonzero morphism occurring in the first case is an epimorphism,
since both $q_{K,K'}$ and multiplication by $\lambda$ are surjective. It follows that, for every collection $\mathcal{K}$ of proper
$B$-submodules of $E$, the collection $\mathcal{E}_{\mathcal{K}} := \{M_K\mid K\in\mathcal{K}\}$ is an epibrick in $\mod \Lambda$. 

A systematic study of mixed semibricks and epibricks, including their
morphism-theoretic structure and their parameter spaces in products
of Grassmannians, lies beyond the scope of the present paper.

\subsection{An almost gentle example}
We conclude with an example showing that our mixed brick framework can determine brick-finiteness outside the gentle and special biserial settings. Let $B=\Bbbk Q$, where $Q$ is the Dynkin quiver of type
$\mathbb{D}_4$ given by
\[
Q:\quad \vcenter{\xymatrix@C=1cm@R=0.7cm{&4\ar[d]&\\1\ar[r]^{\alpha_1}&3&2\ar[l]_{\alpha_2}}}
\]
and let $E=S_1\oplus S_2$. Since $B$ is representation-finite, it is brick-finite and $E$-visible finite. Then, the one-point extension $\Lambda=B[E]$ is obtained by adjoining a new vertex $0$ and two arrows $\rho_1\colon 0\rightarrow1$,  $\rho_2\colon 0\rightarrow2$ subject to the relations $\rho_1\alpha_1=\rho_2\alpha_2=0$. One may check that $\Lambda$ is an almost gentle algebra in the sense of \cite{GS-almost-gentle}. It is not special biserial, since three arrows end at the vertex $3$. In particular, the available brick-finiteness
criteria for gentle (\cite{Plamondon-gentle}) and special biserial (\cite{Schroll-Treffinger-Valdivieso-band}) algebras do not apply directly to this example.

We first claim that the $E$-visible indecomposable $B$-modules are precisely $S_1$ and $S_2$, since they both are injective simple modules. Using Corollary~\ref{cor::mixed-brick-summand}, the $B$-module part of every mixed brick is of the form
\[
N_{\mathbf{m}} = S_1^{\oplus m_1}\oplus S_2^{\oplus m_2},
\qquad
\mathbf{m}=(m_1,m_2)\in\mathbb{Z}_{\geq 0}^2.
\]
Since $S_1$ and $S_2$ are Hom-orthogonal, we have $\Hom_B(E,N_{\mathbf{m}}) \cong \Bbbk^{m_1}\oplus\Bbbk^{m_2}$ and $\End_B(N_{\mathbf{m}}) \cong \Mat_{m_1\times m_1}(\Bbbk)\oplus \Mat_{m_2\times m_2}(\Bbbk)$. Therefore, $h(\mathbf{m})=m_1+m_2$ and $r(\mathbf{m})=m_1^2+m_2^2$. 

We next show that no strict admissible pair exists. Note that
\[
r(\mathbf{m})
=
m_1^2+m_2^2
=
\frac{h(\mathbf{m})^2+(m_1-m_2)^2}{2}
\geq
\frac{h(\mathbf{m})^2}{2}.
\]
If $(\mathbf{m},\ell)$ is admissible and $h(\mathbf{m})\geq3$, then
\[
r(\mathbf{m})-1
\geq
\frac{h(\mathbf{m})^2}{2}-1
>
\frac{h(\mathbf{m})^2}{4},
\]
contradicting admissibility. Hence, every admissible pair satisfies $h(\mathbf{m})\leq2$.
\begin{itemize}
\item Suppose $h(\mathbf{m})=1$. Then, $\mathbf{m}=(1,0)$ or $(0,1)$ and necessarily $\ell=1$. In either case, the pair $(\mathbf{m}, \ell)$ is critical.

\item Suppose $h(\mathbf{m})=2$. If $\mathbf{m}=(2,0)$ or $\mathbf{m}=(0,2)$, then $r(\mathbf{m})-1=3>1\geq \ell(2-\ell)$ and no admissible pair occurs. If $\mathbf{m}=(1,1)$, then $r(\mathbf{m})=2$ and admissibility forces $\ell=1$. It follows that $(\mathbf{m}, \ell)$ is again critical.
\end{itemize}
We have therefore obtained exactly three admissible pairs: $((1,0),1)$, $((0,1),1)$, $((1,1),1)$, and all of them are critical. Our
Theorem~\ref{thm::critical-pair-one-orbit} implies that only finitely
many isomorphism classes of mixed bricks occur in $\mod \Lambda$. It is then obvious that $\Lambda$ is brick-finite.

We finally determine the exact number of mixed bricks. The pair $((1,0),1)$ is realized uniquely: $N_{\mathbf{m}}=S_1$ and $\Gr_1\bigl(\Hom_B(E,S_1)\bigr) = \Gr_1(\Bbbk)$ consists of a single point. Since $\End_B(S_1)=\Bbbk$, the mixed module $(S_1, \Bbbk, \eta)$ with a linear map $0\neq \eta\colon \Bbbk \rightarrow \Bbbk$ is a brick. All such nonzero maps give isomorphic mixed bricks. Similarly, $((0,1),1)$ gives exactly one mixed brick
$(S_2,\Bbbk,\eta)$, where
$\eta\colon\Bbbk\rightarrow\Bbbk$ is nonzero.

It remains to consider $((1,1),1)$. In this case, $N_{\mathbf{m}}=S_1\oplus S_2$, $\Hom_B(E,N_{\mathbf{m}})\cong\Bbbk^2$ and $\End_B(N_{\mathbf{m}}) \cong \Bbbk\oplus\Bbbk$. Write a line in $\Bbbk^2$ as $L_{[a:b]}=\Bbbk(a,b)$ for some $[a:b]\in\mathbb{P}^1$. An endomorphism $(\lambda_1,\lambda_2)\in\Bbbk\oplus\Bbbk$
stabilizes $L_{[a:b]}$ precisely when $(\lambda_1a,\lambda_2b)\in\Bbbk(a,b)$. If $a=0$ or $b=0$, then the stabilizer is the whole algebra
$\Bbbk\oplus\Bbbk$, so the corresponding mixed module is not a
brick. If $ab\neq0$, then the stabilizer condition forces $\lambda_1=\lambda_2$. Consequently, $\Stab_{\End_B(N_{\mathbf{m}})}(L_{[a:b]}) = \Bbbk\operatorname{id}_{N_{\mathbf{m}}}$ if and only if $ab\neq0$. Moreover, $\Aut_B(N_{\mathbf{m}}) = \Bbbk^{\times}\times\Bbbk^{\times}$ acts transitively on the open subset $\{[a:b]\in\mathbb{P}^1\mid ab\neq0\}$. Hence, all scalar-stabilizer lines with $ab\neq0$ define a single isomorphism class of mixed bricks. 

It follows that the three admissible pairs are all realized and that each gives exactly one isomorphism class. Therefore, $\Lambda$ has exactly three isomorphism classes of mixed bricks. 

\ \\
\section*{Acknowledgements}
The author is supported by the National Natural Science Foundation of China No. 12401048 and the Fundamental Research Funds for the Central Universities No. DUT25RC(3)132.

\ \\

\bibliographystyle{alpha}
\bibliography{reference}
\ \\

\end{document}